\documentclass[final,onefignum,onetabnum]{siamart190516}

\usepackage{lipsum}
\usepackage{amsfonts}
\usepackage{graphicx}
\usepackage{epstopdf}
\usepackage{algorithmic}

\usepackage{diagbox}
\usepackage{amsmath}

\usepackage{etoolbox}

\patchcmd{\SetTagPlusEndMark}{$}{}{}{}
\patchcmd{\SetTagPlusEndMark}{$}{}{}{}

\usepackage{amsmath,bbm}
\usepackage{amssymb,multirow,verbatim,comment}
\usepackage{acronym,wrapfig,plain,mathrsfs,enumerate,relsize,color}
\usepackage{tabularx}
\newsiamthm{assumption}{Assumption}
\newsiamremark{example}{Example}
\usepackage{subfig}
\usepackage{boxedminipage}
\usepackage[justification=centering]{caption}
\def\bko{{\rm 1\kern-.17em l}}
\usepackage{wrapfig}
\usepackage{hyperref}
\newcommand{\proj}[2][] {{\mathcal{P}}_{{#1}} {\left(#2\right)}}

\def\argmin{\mathop{\rm argmin}}

\newcommand{\prob}[1]{\mathsf{p}_#1} 
\newcommand{\EXP}[1]{\mathsf{E}\!\left[#1\right]}

\allowdisplaybreaks

\ifpdf
  \DeclareGraphicsExtensions{.eps,.pdf,.png,.jpg}
\else
  \DeclareGraphicsExtensions{.eps}
\fi

\newsiamremark{remark}{Remark}
\newsiamremark{hypothesis}{Hypothesis}
\crefname{hypothesis}{Hypothesis}{Hypotheses}
\newsiamthm{claim}{Claim}

\headers{A method with complexity for optimization with VI constraints}{H.~D.~Kaushik and F.~Yousefian}

\title{A method with convergence rates for optimization problems with variational inequality constraints\thanks{Submitted to the editors DATE. A very preliminary version of this work appeared in the \textit{2019 American Control Conference (ACC)}, IEEE, Philadelphia, PA, USA, 2019, pp. 3420--3425. \url{https://doi.org/10.23919/ACC.2019.8815256}.
\funding{Farzad Yousefian gratefully acknowledges the support of the NSF through CAREER grant ECCS-1944500.}}}

\author{Harshal~D.~Kaushik\thanks{School of Industrial Engineering \& Management, Oklahoma State University, Stillwater, OK 74074, USA
  (\email{harshal.kaushik@okstate.edu}, \email{farzad.yousefian@okstate.edu}).} \and Farzad~Yousefian\footnotemark[2]}

\usepackage{amsopn}

\begin{document}
\sloppy

\maketitle

\begin{abstract}
We consider a class of optimization problems with Cartesian variational inequality (CVI) constraints, where the objective function is convex and the CVI is associated with a monotone mapping and a convex Cartesian product set. This mathematical formulation captures a wide range of optimization problems including those complicated by the presence of equilibrium constraints, complementarity constraints, or an inner-level large scale optimization problem. In particular, an important motivating application arises from the notion of efficiency estimation of equilibria in multi-agent networks, e.g., communication networks and power systems. In the literature, the iteration complexity of the existing solution methods for optimization problems with CVI constraints appears to be unknown. To address this shortcoming, we develop a first-order method called averaging randomized block iteratively regularized gradient (aRB-IRG). The main contributions include: (i) In the case where the associated set of the CVI is bounded and the objective function is nondifferentiable and convex, we derive new non-asymptotic suboptimality and infeasibility convergence rate statements {in a mean sense. We also obtain deterministic variants of the convergence rates when we suppress the randomized block-coordinate scheme}. Importantly, this paper appears to be the first work to provide these rate guarantees for this class of problems. (ii) In the case where the CVI set is unbounded and the objective function is smooth and strongly convex, utilizing the properties of the Tikhonov trajectory, we establish the global convergence of aRB-IRG in an almost sure and a mean sense. We provide the numerical experiments for computing the best Nash equilibrium in a networked Cournot competition model.
\end{abstract}

\begin{keywords}
  first-order methods, variational inequalities, complexity analysis, efficiency of Nash equilibria, iterative regularization, randomized block-coordinate 
\end{keywords}

\begin{AMS}
  65K15, 49J40, 90C33, 91A10, 90C06
\end{AMS}

\section{Introduction}
Traditionally, the mathematical models and algorithms for constrained optimization have been much focused on the cases where the functional constraints are in the form of inequalities, equalities, or easy-to-project sets. However, in a breadth of emerging applications in the control theory and economics, the system constraints are too complex to be characterized in those forms. This may arise in several network application domains where the optimization model is complicated by the presence of equilibrium constraints, complementarity constraints, or an inner-level large scale optimization problem. Accordingly, the goal in this paper lies in addressing the following constrained optimization problem:
\begin{equation}\label{prob:main}
\tag{$\text{P}^f_{\text{VI}}$} 
\begin{aligned}
& {\text{minimize}}
& & f(x) \\
& \text{subject to}
& & x \in \text{SOL}(X,F).
\end{aligned}
\end{equation}
Here, $f:\mathbb{R}^n \to \mathbb{R}$ is a convex function and $X \subseteq \mathbb{R}^n$ is given as a Cartesian product, i.e., $ X \triangleq \prod_{i = 1}^{d}  X_i $, where $X_i \subseteq\mathbb{R}^{n_i}$ is convex for all $i=1,\dots, d$ and $\sum_{i=1}^dn_i=n$. We consider $F: X \to \mathbb{R}^n$ to be a monotone mapping. The term $\text{SOL}(X,F)$ denotes the solution set of the variational inequality $\text{VI}(X,F)$ defined as follows:  A vector $x \in X$ is said to be a solution to $\text{VI}(X,F)$ if for any $y \in X$, we have $F(x)^T(y-x) \geq 0$. Variational inequalities, first introduced in late 1950s, are an immensely powerful mathematical tool that can serve as a unifying framework for capturing a wide range of applications arising in operations research, finance, and economics (cf. \cite{FacchineiPang2003, Rockafellar98,Lan-VI-13,wang2015}). Importantly, as it will be discussed shortly, the block structure of the set $X$ allows for addressing Nash games as well as high-dimensional optimization problems. We note that the problem \cref{prob:main} can represent a variety of the standard problems in optimization and VI regimes. For example, when $F(x) := 0_{n}$, the problem \cref{prob:main} is equivalent to the canonical optimization problem $\min_{x \in X} f(x)$. Also, when $f(x):=0$, the problem \cref{prob:main} is equivalent to solving $\text{VI}(X,F)$. More detailed examples that can be reformulated as the problem \cref{prob:main} are presented below.

\subsection{Motivating examples}
 \begin{example}[Efficiency estimation of equilibria]\label{ex:efficiency}
The main motivating application arises from the notion of \textit{efficiency of equilibria}  in multi-agent networks, including communication networks and power systems. In the noncooperative regimes, the system behavior is governed by a collection of decisions (i.e., equilibrium) made by a set of independent and {self-interested} agents. As a result of this noncooperative behavior (i.e., \textit{game}) among the agents, the global performance of the system may become worse than the case where the agents cooperatively seek an \textit{optimal} decision. A well-known example is the Prisoners' Dilemma where the costs of the players incurred by the Nash equilibrium are superior to their costs when they cooperate~\cite{OsborneRubinstein1994}. Indeed, it has been well-received in economics and computer science communities that Nash equilibria of a game may not attain full efficiency. This perception has led to a surge of research for understanding the quality of an equilibrium in noncooperative games. In particular, addressing this question becomes imperative for network design~\cite{PoS08} in the areas of routing~\cite{Correa04} and load balancing~\cite{Rough04}. In such networks, a {protocol designer} seeks \textit{the best equilibrium} with respect to a global performance measure, i.e., the function $f$ in \cref{prob:main}. To this end, the notion of  \textit{price of stability} is defined as the ratio of the best objective function value over the set of equilibria to the best objective function value under no competition \cite{NisanBook2007}. 
 In regard to the choice of the objective function $f$ in \cref{prob:main}, different approaches have been considered, including the \textit{utilitarian} function and the \textit{egalitarian} function \cite{NisanBook2007}. In the utilitarian approach, $f$ is defined as the summation of the individual objective functions of the agents, while in the egalitarian approach, the maximum of the individual cost functions is considered. In the context of network resource allocation where a monetary value is measured, the utilitarian approach is also referred to as Marshallian aggregate surplus (e.g., see \cite{JohariThesis}). In the following, we describe the details for the problem of selecting the best equilibrium in Nash games, where we employ the utilitarian approach. Consider a canonical Nash game among $d$ players where the $i^{\text{th}}$ player is associated with a strategy $x^{(i)} \in X_i\subseteq \mathbb{R}^{n_i}$ and a cost function $ g_i\left(x^{(i)};x^{(-i)}\right) $, where $x^{(-i)}$ denotes the collection of actions of other players. Nash games arise in a wide range of problems including communication networks~\cite{alpcan02game,alpcan03distributed,yin09nash2}, cognitive radio networks~\cite{aldo1,scutari10monotone,KoshalNedichShanbhag2013}, and power markets~\cite{KShKim11,KShKim12,SIGlynn11}. The game is defined formally as the following collection of problems for all $i =1,\ldots,d$:
\begin{align}\label{eqn:game}
& \text{minimize}_{x^{(i)}} \quad g_i\left(x^{(i)};x^{(-i)}\right) \tag*{P$_i\left(x^{(-i)}\right)$}\\
& \text{subject to}  \quad x^{(i)} \in {X}_i\notag.
\end{align}
A Nash equilibrium (NE) is a tuple of strategies $x^*\triangleq \left({x^*}^{(1)};{x^*}^{(2)};\ldots;{x^*}^{(d)}\right)$ where no player can obtain a lower cost by deviating from his own strategy, given that the strategies of the other players remain unchanged. It is known that (cf. {Proposition 1.4.2}  \cite{FacchineiPang2003}) when for all $i$, $X_i$ is a closed convex set and $g_i$ is a differentiable convex function with respect to $x^{(i)}$, the resulting equilibrium conditions of the Nash game given by \cref{eqn:game}, are compactly captured by a Cartesian $\text{VI}(X,F)$ where $X \triangleq \prod_{i=1}^d X_i$ and $F(x) \triangleq  (F_1(x);\ldots; F_d(x))$ with $F_i(x)\triangleq \nabla_{x^{(i)}}  g_i\left(x^{(i)};x^{(-i)}\right) $. The set $\text{SOL}(X,F)$ will then represent the set of Nash equilibria to the game \cref{eqn:game}. The best NE problem employing the utilitarian approach is formulated as follows:
\begin{equation}\label{prob:best_NE}
\begin{aligned}
& {\text{minimize}} 
& & \sum_{i=1}^d g_i\left(x^{(i)};x^{(-i)}\right) \\
& \text{subject to}
& & x \in \text{SOL}\left( \prod_{i=1}^d X_i,\left(\nabla_{x^{(1)}}  g_1\left(x^{(1)};x^{(-1)}\right); \ldots; \nabla_{x^{(d)}}  g_d\left(x^{(d)};x^{(-d)}\right)\right)\right).
\end{aligned}
\end{equation}
In \cref{sec:experiments}, we solve the model \cref{prob:best_NE} for a class of networked Nash-Cournot games. 
\end{example}

\begin{example}[High-dimensional constrained convex optimization]\label{ex:high_d_opt} 
Another class of problems that can be captured by the model \cref{prob:main} is as follows:
\begin{equation}\label{prob:subclass_constrained_opt}
\begin{aligned}
& {\text{minimize}} & & f(x) \\
& \text{subject to}  & &h_j(x) \leq 0\quad \hbox{for all } j=1,\ldots,J\\
& & & Ax=b, \quad   x \in X \triangleq \prod\nolimits_{i=1}^dX_i,
\end{aligned}
\end{equation}
where $x \in \mathbb{R}^n$, $A \in \mathbb{R}^{m\times n}$, $b \in \mathbb{R}^m$, $h_j:\mathbb{R}^n\to \mathbb{R}$ for all $j$, and $n$ is possibly very large. In the following, we show a case where the problem \cref{prob:subclass_constrained_opt} can be cast as \cref{prob:main}: 
\begin{lemma}\label{Lemma 1.3} Let the problem \cref{prob:subclass_constrained_opt} be feasible and $h_j(x)$ be a continuously differentiable convex function for all $j$. Let the set $X_i \in \mathbb{R}^{n_i}$ be nonempty, closed, and convex for all $i$. Then, the problem \cref{prob:subclass_constrained_opt} is equivalent to \cref{prob:main} where $F:\mathbb{R}^n \to \mathbb{R}^n$ is defined as    
$F(x) \triangleq A^T(Ax-b)+\sum_{j=1}^J\max\{0,h_j(x)\}\nabla h_j(x)$.
\end{lemma} 
\begin{proof}
See \cref{app:equiv_convexprogram}.
\end{proof}
\end{example}

\begin{example}[Optimization problems with complementarity constraints]\label{ex:opt_compl_const} Another class of problems that can be addressed in this work is as follows:
\begin{equation}\label{prob:subclass_constrained_compl}
\begin{aligned}
& {\text{minimize}} & & f(x) \\
& \text{subject to}  & &  x^TF(x) =0, \quad x\geq 0, \quad F(x) \geq 0,
\end{aligned}
\end{equation}
where $F:\mathbb{R}^n\to\mathbb{R}^n$ is a mapping. Then, problem \cref{prob:subclass_constrained_compl} can be cast as \cref{prob:main} where the set $X$ is the nonnegative orthant, i.e., $X\triangleq \mathbb{R}^n_{+}$ (see Proposition 1.1.3 in~\cite{FacchineiPang2003}).
\end{example}
\begin{example}[Optimization problems with nonlinear equality constraints]\label{ex:opt_nl_const}
Consider the following optimization problem:
\begin{equation}\label{prob:subclass_equ_constrained}
\begin{aligned}
& {\text{minimize}} & & f(x) \\
& \text{subject to}  & &  F(x) =0.
\end{aligned}
\end{equation}
 where $F:\mathbb{R}^n \to \mathbb{R}^n$ is a mapping. Defining $X \triangleq \mathbb{R}^n$, $\text{SOL}(X,F)$ is equal to the feasible solution set of the problem \cref{prob:subclass_equ_constrained}. This implies that the problem \cref{prob:subclass_equ_constrained} can be captured by the formulation \cref{prob:main}.
 \end{example}

\subsection{Existing methods and the research gap} We first begin by providing a brief overview of the solution methods for addressing a VI problem. Starting from the seminal work of Lemke and Howson \cite{LemkeHowson1964} and Scarf \cite{Scarf1967}, who developed the first solution methods for computing equilibria, in the past few decades, there has been a surge of research on the development and analysis of the computational methods for solving VIs. Perhaps this interest lies in the strong interplay between the VIs and the formulation of optimization and equilibrium problems arising in many communication and networking problems 
\cite{ScutariPalomarFacchineiPang2010}. Korpelevich's celebrated extragradient method~\cite{korp76} and its extensions~\cite{Nemirovski2004,Nem11,CensorGibaliReich2011,IusemNasri2011,CensorGibaliReich2012,YousefianNedichShanbhag2014,IusemJofreOliveiraThompson2016,ChenLanOuyang2017,FarzadSetValued18}
were developed which require weaker assumptions than their gradient counterparts. In the past decade, there has been a trending interest in addressing VIs in the stochastic regimes. Among these, Jiang and Xu \cite{JiangXu2008} developed the stochastic approximation methods for solving VIs with strongly monotone and smooth mappings. This work was later extended to the case with merely monotone mappings \cite{KoshalNedichShanbhag2013,KannanShanbhag2012,IusemJofreThompson2019} and nonsmooth mappings \cite{FarzadMathProg17}. 
\begin{algorithm} [h]
	\caption{The existing SR scheme for solving problem \cref{prob:main} when $f:=\frac{1}{2}\|\cdot\|^2$}
	\begin{algorithmic}[1]
		\STATE\textbf{Input:} {Set $X$, mapping $F$,  and an initial regularization parameter $\eta_0> 0$.}
		\FOR{t = 0, 1, \dots}
		\STATE Compute $x_{\eta_t}^*$ defined as $x_{\eta_t}^* \in \text{SOL}\left(X,F+\eta_t\textbf{I}_n\right)$.
		\STATE Update $\eta_{t}$ to $\eta_{t+1}$ such that $\eta_{t+1}<\eta_{t}$.
		\ENDFOR
	\end{algorithmic}\label{alg:two-loop}
\end{algorithm}

Despite much advances in the theory and algorithms for VIs, solving the problem \cref{prob:main} has remained challenging. To the best of our knowledge, the computational complexity of the existing solution methods for addressing \cref{prob:main} is unknown. In addressing the standard constrained optimization problems, Lagrangian duality and relaxation rules have often proven to be very successful \cite{BertsekasNLPBook2016}. However, when it comes to solving \cref{prob:main}, the duality theory cannot be practically employed. This is primarily because unlike in the standard constrained optimization problems where the objective function provides a metric for distinguishing solutions, there is no immediate analog in the VI problems. Inspired by the contributions of Andrey Tikhonov in 1980s on addressing illposed optimization problems, the existing methods for solving \cref{prob:main} share in common a sequential regularization (SR) scheme presented by \cref{alg:two-loop}. The SR scheme is a two-loop framework where at each iteration, given a fixed parameter $\eta_t$, a regularized VI denoted by $\text{VI}\left(X,F+\eta_t\mathbf{I}_n\right)$ is required to be solved. In the special case where $f(x) := \frac{1}{2}\|x\|^2$, it can be shown when $\eta_t \to 0$, under the monotonicity of the mapping $F$ and closedness and convexity of the set $X$, any limit point of the \textit{Tikhonov trajectory} denoted by $\{x_{\eta_t}^*\}$, where $ x_{\eta_t}^*\in \text{SOL}\left(X,F+\eta_t\textbf{I}_n\right)$, converges to the least $\ell_2$--norm vector in $\text{SOL}(X,F)$ (cf. Chapter 12 in  \cite{FacchineiPang2003}). The SR approach is associated with two main drawbacks: (i) It is a computationally inefficient scheme, as it requires solving a series of increasingly more difficult VI problems. (ii) The iteration complexity of the SR scheme in addressing the problem \cref{prob:main} is unknown. Accordingly, the main goal in this work lies in the development of an efficient scheme equipped with computational complexity analysis for solving the problem \cref{prob:main}.

\subsection{Summary of contributions} Our main contributions are as follows: 

\noindent \textbf{(i)} \textit{Development of a single timescale method equipped with convergence rate guarantees}: In addressing \cref{prob:main}, we develop an efficient first-order method called averaging randomized block iteratively regularized gradient (aRB-IRG). The proposed method is single timescale in the sense that, unlike the SR approach, it does not require solving a VI at each iteration. Instead, it only uses evaluations of the mapping $F$ and the subgradient of the objective function $f$ at each iteration. In the first part of the paper, we consider the case where the set $X$ is bounded. We let $f$ be a subdifferentiable merely convex function and $F$ be a monotone mapping. In \cref{thm:a-IRG_rate_results}, we derive a suboptimality convergence rate in terms of {the expected value of the objective function}. We also derive a convergence rate for the infeasibility that is characterized by {the expected value} of a dual gap function. {We also derive deterministic variants of the aforementioned convergence rates when we suppress the randomized block-coordinate scheme.} In the second part of the paper, we consider the case where the set $X$ is unbounded and $f$ is smooth and strongly convex. Utilizing the properties of the Tikhonov trajectory, we establish the global convergence of the scheme in an almost sure and a mean sense. To the best of our knowledge, this work appears to be the first paper that provides the two rate statements for problems of the form \cref{prob:main}. In particular, the complexity analysis in this work {contributes to the existing convergence theory} in several previous papers including~\cite{XuViscosity2004,KoshalNedichShanbhag2013,KannanShanbhag2012,FarzadMathProg17,FarzadHarshalACC19,FarzadSIOPT20}. Moreover, in the special case where the VI constraints represent the optimal solution set of an optimization problem, \cref{prob:main} captures a class of bilevel optimization problems. This class of problems has been studied in a number of recent papers in deterministic~\cite{Solodov2007,BeckSabach2014,SabachShtern2017,GarrigosRosascoVilla2018}, stochastic \cite{FarzadOMS19}, and distributed regimes \cite{FarzadPushPull2020}. However, the complexity analysis in the aforementioned papers lacks a suboptimality rate, or lacks an infeasibility rate, or requires much stronger assumptions such as strong convexity and smoothness of $f$. 

\noindent \textbf{(ii)} \textit{Advancing the convergence rate properties of the randomized block-coordinate schemes}: Block-coordinate schemes, and specifically their randomized variants, have been widely studied in addressing the standard optimization problems (e.g., see \cite{Nesterov2012,RicktarikTakac2014,ShwartzZhang2013,Dang15,FarzadSetValued18}). However, in addressing VI problems, there are only a handful of recent papers, including \cite{LeiShanbhag2020,FarzadSetValued18}, that employ this technique and are equipped with rate guarantees. The aforementioned papers address standard VI problems that can be viewed a special case of the model \cref{prob:main} where $f(x):=0$.  In this work, we extend the convergence and rate analysis of the randomized block-coordinate schemes to the much broader regime of optimization problems with CVI constraints. 

\textbf{Outline of the paper:} The paper is organized as follows. The proposed algorithm is presented in \cref{sec:alg}. The complexity analysis is provided in \cref{sec:conv_rate_ana}. In \cref{sec:conv_unbounded}, we provide the convergence analysis when the set $X$ is unbounded. The experimental results are in \cref{sec:experiments}, and the conclusions follow in
\cref{sec:conclusions}. 

\textbf{Notation and preliminary definitions:} Throughout, a vector $ x \in \mathbb{R}^n$ is assumed to be a column vector and $ x^T$ denotes the transpose of $x$. We use $x^{(i)}\in \mathbb{R}^{n_i}$ to denote the $i^{\text{th}}$ block-coordinate of vector $x$ where $x =\left(x^{(1)};\ldots;x^{(d)}\right)$ and $\sum_{i=1}^dn_i=n$. The Euclidean norm of a vector $x$ is denoted by $\|x\|$, i.e.,  $\| x\|\triangleq \sqrt{ x^T  x}.$  For a mapping $F:\mathbb{R}^n \to \mathbb{R}^n$, we denote the $i^{\text{th}}$ block-coordinate of $F$ by $F_i:\mathbb{R}^n \to \mathbb{R}^{n_i}$, i.e., $F(x) =\left(F_1(x);\ldots;F_d(x)\right)$. A mapping $F:\mathbb{R}^n \to \mathbb{R}^n$ is said to be monotone on a convex set $X \subseteq \mathbb{R}^n$ if for any $x,y \in X$, we have $(F(x)-F(y))^T(x-y)\geq 0$. The mapping $F$ is said to be $\mu$--strongly monotone on a convex set $X \subseteq \mathbb{R}^n$ if $\mu>0$ and for any $x,y \in X$, we have $(F(x)-F(y))^T(x-y)\geq \mu\|x-y\|^2$. Also, $F$ is said to be Lipschitz with parameter $L>0$ on the set $X$ if for any $x,y \in X$, we have $\|F(x)-F(y)\|\leq L\|x-y\|$. A continuously differentiable function $f:\mathbb{R}^n \to \mathbb{R}$ is called $\mu$--strongly convex on a convex set $X$ if $f(x) \geq f(y)+\nabla f(y)^T(x-y)+\frac{\mu}{2}\|x-y\|^2$. Function $f$ is $\mu$--strongly convex if and only if $\nabla f$ is $\mu$--strongly monotone on $X$. For a convex function $f$ with the domain $\text{dom}(f)$, the subgradient of $f$ at $x \in \text{dom}(f)$ is denoted by $\tilde \nabla f(x)$ and it satisfies $f(x)+\tilde \nabla f(x)^T(y-x) \leq f(y)$ for all $y \in \text{dom}(f)$. The subdifferential set of $f$ at $x$ is the set of all subgradients of $f$ at $x$ and is denoted by $\partial f(x)$. The Euclidean projection of vector $s$ onto a convex set $ S$ is denoted by $ \proj[S]{ s} $, where $\proj[S]{ s}\triangleq  \argmin_{ y \in  S}\| s- y\|$. 
We use $\mathbf{I}_n$ to denote the identity matrix of size $n\times n$. The probability of an event $Z$ is denoted by $\text{Prob}(Z)$ and the expectation of a random variable $z$ is denoted by $\EXP{z}$. We use $\mathbb{R}^n_+$ and $\mathbb{R}^n_{++}$ to denote $\{x \in \mathbb{R}^n \mid x \geq 0\}$ and  $\{x \in \mathbb{R}^n \mid x > 0\}$, respectively.

\section{Outline of the algorithm}\label{sec:alg}
In this section, we state the main assumptions and present the proposed scheme for solving the optimization problem \cref{prob:main}. 
\begin{assumption}\label{assum:problem} Consider the problem \cref{prob:main} under the following conditions:

\noindent (a) The set $X_i$ is nonempty, closed, and convex for all $i=1,\dots, d$.

\noindent (b) The function $f$ is convex and has bounded subgradients over the set $X$.

\noindent  (c) The mapping $F:\mathbb{R}^n \to \mathbb{R}^n$ is continuous, monotone, and bounded over the set $X$. 

\noindent (d) The optimal solution set of problem \cref{prob:main} in nonempty.
\end{assumption}
{\cref{assum:problem}(b) implies that $f$ is Lipschitz continuous over the set $X$. Under this assumption}, we address a broad class of problems of the form \cref{prob:main} where the objective function is possibly nondifferentiable and nonstrongly convex. In the following, we discuss the conditions under which \cref{assum:problem}(d) is satisfied. 
\begin{remark}[Existence of an optimal solution]
Suppose \cref{assum:problem}(a), (b), and (c) hold. The existence of an optimal solution to the problem \cref{prob:main} can be established under different conditions. We provide two instances as follows: (i) Suppose there {exists} a vector $\bar x \in X$ such that the set $\bar X\triangleq \{x \in X: \ F(x)^T(x-\bar x)\leq 0\}$ is bounded. Then, from Proposition 2.2.3 in~\cite{FacchineiPang2003}, $\text{SOL}(X,F)$ is nonempty and compact. Consequently, the Weierstrass' Theorem implies the existence of an optimal solution to the problem \cref{prob:main}. (ii) Suppose the set $X$ is compact. Then, from Corollary 2.2.5 in~\cite{FacchineiPang2003}, the set $\text{SOL}(X,F)$ is nonempty and compact. Again, \cref{assum:problem}(d) is guaranteed by the Weierstrass' Theorem.
\end{remark}
Throughout, we let $C_F>0$ denote the bound on the Euclidean norm of the mapping $F$, i.e., $\|F(x)\| \leq C_F$ for all $x \in X$. Also, we let $C_f>0$ denote the bound on the norm of the subgradients of $f$, i.e., $\|\tilde \nabla f (x)\| \leq C_f$ for all $\tilde \nabla f(x) \in \partial f(x)$ and $x \in X$.
\begin{algorithm} [h]
	\caption{aRB-IRG}
	\begin{algorithmic}[1]
		\STATE\textbf{Input:} A random initial point $x_0 \in X$, $\bar{x}_0 := x_0$, initial stepsize $\gamma_{0} > 0$, initial regularization parameter $\eta_{0} > 0$, a scalar $0\leq r<1$, and $S_0 := \gamma_0^r$.
		\FOR{k = 0, 1, \dots}
		\STATE Generate a realization of random variable $i_k$ according to \cref{assum:random sample}.
		\STATE Evaluate $F_{i_k}(x_k)$ and $\tilde \nabla_{i_k} f(x_k)$ where $\tilde \nabla f(x_k) \in \partial f(x_k)$.

		\STATE Update $x_k$ as follows:
		\begin{align}\label{equ:update_rule_aRBIRG} 
		{x_{k+1}^{(i)}} :=
		\begin{cases}
		\proj[X_{i_k}]{ x_{k}^{(i_k)}-\gamma_{k}\left(F_{i_k}\left( x_{k}\right) + \eta_{k} \tilde \nabla_{{i_k}} f\left( x_{k}\right)\right)}& \text{ if } i = i_k, \\
		x_k^{(i)} & \text{ if } i \neq i_k. 
		\end{cases}
		\end{align}
		\STATE Obtain $\gamma_{k+1}$ and $\eta_{k+1}$ (cf. \cref{thm:a-IRG_rate_results} and \cref{thm:convergence_RB-IRG} for the update rules).

		\STATE Update the averaged iterate $\bar{x}_{k}$ as follows:
			\begin{align}\label{eqn:ave_step_of_alg}
			S_{k+1} &:= S_k + \gamma_{k+1}^r,\quad	\bar{x}_{k+1}:= \frac{S_k\bar{x}_{k}+\gamma_{k+1}^rx_{k+1}}{S_{k+1}}.
			\end{align}
		\ENDFOR
	\end{algorithmic} \label{alg:aRB-IRG}
\end{algorithm}
The outline of the proposed method is presented by \cref{alg:aRB-IRG}. At iteration $k$, a block-coordinate index $i_k$ is selected randomly as follows: 
\begin{assumption}[Block-coordinate selection rule] \label{assum:random sample}
At each iteration $k\geq 0$, the random variable $i_k$ is generated from an independent and identically distributed discrete probability distribution such that $\mathrm{Prob}\left(i_k=i\right) = \prob{i}$ where $ \prob{i}>0$ for $i \in \{1,\ldots,d\}$ and $\sum_{i=1}^{d}\prob{i}=1$.
\end{assumption}
Then, the $i_k^\text{th}$ block-coordinate of the iterate $x_k$ is updated using \cref{equ:update_rule_aRBIRG}. Here, $\gamma_k$ denotes the stepsize at iteration $k$ and $\eta_k$ denotes the regularization parameter at iteration $k$. We note that these sequences are updated iteratively. Here, we incorporate {the} information of the mapping $F$ and the subgradient mapping $\tilde \nabla f$ by employing an iterative regularization scheme. {At each iteration, a projection operation is performed onto a randomly selected set $X_{i_k}$. For this reason, the proposed algorithm finds some relevance with the prior study on randomized projection methods such as~\cite{Nedich2011}.} We will show that the convergence and rate analysis of the proposed method mainly rely on the choices of $\{\gamma_k\}$ and $\{\eta_k\}$. Accordingly, one key research objective in this section will center around the development of suitable update rules for the two sequences so that we can establish the convergence and derive rate statements. To obtain the rate results, we employ an averaging step using the equations given by \cref{eqn:ave_step_of_alg}, where the sequence $\{\bar{x}_k\}$ is obtained as a weighted average of  $\{x_k\}$. The averaging weights are characterized by the stepsize $\gamma_k$ and a scalar $r \in \mathbb{R}$. It will be shown that the rate results can be provided when $0\leq r<1$ (cf. \cref{thm:a-IRG_rate_results}). 

{\begin{remark}
Importantly, unlike \cref{alg:two-loop}, \cref{alg:aRB-IRG} is a single timescale scheme that does not require solving any inner-level VI problem. In particular, the update rule given by \cref{equ:update_rule_aRBIRG} mainly requires evaluations of random blocks of the mappings $F$ and $\tilde \nabla f$. For this reason, \cref{equ:update_rule_aRBIRG} is computationally more efficient than the step $3$ in \cref{alg:two-loop}.
\end{remark}} 

\subsection{Preliminaries} 
In the following, we provide some definitions and preliminary results that will be used to analyze the convergence of \cref{alg:aRB-IRG}. 
\begin{definition}[Distance function]\label{def:err_D}
	For any $x,y \in \mathbb{R}^n$, function $\mathcal{D}(x,y)$ is defined as $
	\mathcal{D}(x,y)\triangleq \sum_{i=1}^{d} \prob{i}^{-1} \left\Vert\left. x^{(i)}-y^{(i)} \right\Vert\right.^2$, where $\prob{i}$ is given by \cref{assum:random sample}.
\end{definition}
\begin{remark} \label{rem:err_D}
Under \cref{assum:random sample}, we can relate the distance function $\mathcal{D}$ with the $\ell_2$--norm as follows: $\prob{{min}}	{\mathcal{D}(x,y)} \leq \|x-y\|^2 \leq\prob{{max}} {\mathcal{D}(x,y)}$ for all $x,y \in \mathbb{R}^n$, where  $\prob{{min}}\triangleq \min_{1\leq i \leq d}\{\prob{i}\}$ and $\prob{{max}}\triangleq \max_{1\leq i \leq d}\{\prob{i}\}$.
\end{remark}
One of the main challenges {in} the convergence analysis of computational methods for solving VI problems lies in the lack of access to a standard metric to quantify the quality of the solution iterates. {T}his is in contrast with solving the standard optimization problems where the objective function can serve as an immediate performance metric for the underlying algorithm. Addressing this challenge in the literature of VI problems has led to the study of so-called \textit{gap functions} (cf.~\cite{FacchineiPang2003,FarzadMathProg17}). Of these, in the analysis of this section, we use the dual gap function defined as follows:
\begin{definition}[The dual gap function \cite{MarcotteZhu1998}]\label{def:gap}
Let a nonempty closed set $X\subseteq \mathbb{R}^n$ and a mapping $F:X\rightarrow\mathbb{R}^n$ be given. Then, for any $x \in X$, the dual gap function $\mathrm{GAP}:X\rightarrow \mathbb{R}\cup \{+\infty\} $ is defined as $ \mathrm{GAP}(x) \triangleq \sup_{y\in X} F(y)^T(x-y)$.
\end{definition}
\begin{remark} 
{When $X \neq \emptyset$, \cref{def:gap} implies} that the dual gap function is nonnegative over $X$. It is also known that when $F$ is continuous and monotone and the set $X$ is closed and convex, $\text{GAP}(x^*)=0$ if and only if $x^* \in \text{SOL}(X,F)$ (cf.~\cite{Nem11}). Thus, we conclude that under \cref{assum:problem}, the dual gap function is well-defined.
\end{remark}
\begin{definition}[Regularized mapping]\label{def:reg_map}
Given a vector $x \in X$, a subgradient $\tilde \nabla f(x) \in \partial f(x)$, and an integer $k\geq 0$, the regularized mapping $G_k: X \to \mathbb{R}$ is defined as $G_k(x) \triangleq F(x)+\eta_k\tilde \nabla f(x)$. The $i^\text{th}$ {block-coordinate} of $G_k$ is denoted by $G_{k,i}$.
\end{definition}
\begin{definition}[History of the method]\label{def:Fk}
Throughout, we let the history of the algorithm to be denoted by $\mathcal{F}_k \triangleq \{x_0,i_0,i_1,\ldots,i_{k-1}\}$ for $k\geq 1$, with $\mathcal{F}_0 \triangleq \{x_0\}$.
\end{definition}
Next, we show that $\bar x_k$ generated by \cref{alg:aRB-IRG} is a well-defined weighted average.\begin{lemma}[Weighted averaging]\label{Lemma 2.10}
Let $\{\bar x_k\}$ be generated by \cref{alg:aRB-IRG}. Let us define the weights $\lambda_{k,N} \triangleq \frac{\gamma_k^r}{\sum_{j=0}^N \gamma_j^r}$ for $k \in \{0,\ldots, N\}$ and $N\geq 0$. Then, for any $N\geq 0$, we have $\bar{x}_{N} = \sum_{k=0}^N \lambda_{k,N} x_k$. Also, when $X$ is a convex set, we have $\bar x_N \in X$.
\end{lemma}
\begin{proof}
See \cref{app:ave_x_second_formula}. 
\end{proof}
In the following, we define two terms that characterize the error between the true maps with their randomized block variants. 
\begin{definition}[Randomized block error terms]\label{def:deltas}
Let ${\mathbf{U}}_i \in \mathbb{R}^{n\times n_i}$ for $i=1,\ldots,d$ be the collection of matrices such that $\mathbf{I}_n = \left[{\mathbf{U}}_1,\ldots,{\mathbf{U}}_d\right] \in \mathbb{R}^{n\times n}$. Consider the following definitions for $k\geq 0$: 
\begin{align}\label{eqn:deltas}
\Delta_k\triangleq F(x_k) - \prob{{i_k}}^{-1}{\mathbf{U}}_{i_k}F_{i_k}(x_k), \qquad \delta_k\triangleq \tilde \nabla f(x_k) - \prob{{i_k}}^{-1}{\mathbf{U}}_{i_k} \tilde \nabla_{i_k}f (x_k).
\end{align}
\end{definition}
\begin{lemma}[Properties of $\Delta_k$ and $\delta_k$]\label{Lemma 2.12}
Consider \cref{def:deltas}. We have:\\
\noindent (a)  $\EXP{\Delta_k\mid \mathcal{F}_k}=\EXP{\delta_k\mid \mathcal{F}_k}=0$.\\
\noindent (b) $\EXP{\|\Delta_k\|^2\mid \mathcal{F}_k} \leq  \left(\prob{{min}}^{-1}-1\right)C_F^2$ and $\EXP{\|\delta_k\|^2\mid \mathcal{F}_k} \leq  \left(\prob{{min}}^{-1}-1\right)C_f^2$.
\end{lemma}
\begin{proof} 
See \cref{app:deltas}
\end{proof}
We will use the next result in deriving the suboptimality and infeasibility rate results.\begin{lemma}[Bounds on the harmonic series]\label{Lemma 2.13}
Let $0\leq \alpha <1$ be a given scalar. Then, for any integer $N \geq 2^{\frac{1}{1-\alpha}}-1$, we have 
$\frac{(N+1)^{1-\alpha}}{2(1-\alpha)}  \leq \sum_{k=0}^{N}\frac{1}{(k+1)^\alpha}\leq \frac{(N+1)^{1-\alpha}}{1-\alpha}$.
\end{lemma}
\begin{proof}
See \cref{app:harmonic_bounds}
\end{proof}

\section{Convergence rate analysis}\label{sec:conv_rate_ana}
In the following result, we derive an inequality that will be later used to construct bounds on the objective function value and the dual gap function at the averaged sequence generated by \cref{alg:aRB-IRG}.
\begin{lemma}\label{lem:prebound_gap}
Consider the sequence $\{x_k\}$ in \cref{alg:aRB-IRG}. Suppose $\{\gamma_k\}$ and $\{\eta_k\}$ are strictly positive sequences. Let \cref{assum:problem} and \cref{assum:random sample} hold. Let the auxiliary sequence $\{u_k\} \subset X$ be defined as 
		{$u_{k+1} \triangleq \mathcal{P}_X\left(u_k-\gamma_k\left(\Delta_k+\eta_k\delta_k\right)\right)$},
where $u_0 \in X$ is an arbitrary vector. Then, for {all $y\in X$ and all $k \geq 0$, we have}:
\begin{align}\label{eqn:prebound_gap}
&\gamma_k^rF(y)^T(x_k-y) +\gamma_k^r\eta_k\tilde \nabla f(x_k)^T(x_k-y)\leq \frac{\gamma_k^{r-1}}{2}\left(\mathcal{D}\left(x_{k}, y\right) +\|u_k-y\|^2\right) \nonumber\\
&-\frac{\gamma_k^{r-1}}{2}\left(\mathcal{D}\left(x_{k+1}, y\right) +\|u_{k+1}-y\|^2\right){+ \gamma_k^r(x_k-u_k)^T\left(\Delta_k+\eta_k\delta_k\right)}\nonumber\\
& +\gamma_k^{r+1}\left(\|\Delta_k\|^2+\eta_k^2\|\delta_k\|^2\right)+ 0.5\prob{{i_k}}^{-1}\gamma_k^{1+r}\left\Vert\left.G_{k,i_k}\left(x_k\right)\right\Vert\right.^2.
\end{align}
\end{lemma}

\begin{proof} Let $k\geq 1$ be fixed. From \cref{def:err_D} and \cref{equ:update_rule_aRBIRG}, for any $y \in X$, we have:
	\begin{align}\label{eqn:D_split}
	\mathcal{D} \left(x_{k+1}, y\right)  =   \prob{{i_k}}^{-1} \left\Vert\left.  x_{k+1}^{(i_k)}-  y^{(i_k)} \right\Vert\right.^2+\sum_{i=1, \thinspace	 i \neq i_k}^{d} \prob{i}^{-1} \left\Vert\left.  x_{k}^{(i)}-  y^{(i)} \right\Vert\right.^2.
	\end{align}Next, we find a bound on the term $\left\Vert\left.  x_{k+1}^{(i_k)}- y^{(i_k)} \right\Vert\right.^2$. From the block structure of $X$, we have $y^{(i_k)}\in X_{i_k}$. Invoking the nonexpansiveness property of the projection mapping, the update rule \cref{equ:update_rule_aRBIRG}, \cref{def:reg_map}, and the preceding relation, we obtain:
		\begin{align*} 
	\left\Vert\left. x_{k+1}^{(i_k)}-  y^{(i_k)}\right\Vert\right.^2 \leq\left\Vert\left.  x_{k}^{(i_k)}-\gamma_{k}G_{k,i_k} \left( x_{k}\right) -y^{(i_k)}\right\Vert\right.^2.
	\end{align*}
		Combining the preceding relation with \cref{eqn:D_split}, we obtain:
	\begin{align*}
	 	\mathcal{D}\left(x_{k+1}, y\right)& \leq 
	 	 \sum_{i=1, \thinspace	 i \neq i_k}^{d} \prob{i}^{-1} \left\Vert\left.  x_{k}^{(i)}-  y^{(i)} \right\Vert\right.^2 +  \prob{{i_k}}^{-1}\left\Vert\left.  x_{k}^{(i_k)}-  y^{(i_k)} \right\Vert\right.^2 \nonumber \\
& - 2\,\prob{{i_k}}^{-1}\gamma_k\left( x^{(i_k)}_{k} - y^{(i_k)}\right)^TG_{k,i_k}\left(x_k\right)+ \prob{{i_k}}^{-1}\gamma_k^2 \left\Vert\left.G_{k,i_k}\left(x_k\right)\right\Vert\right.^2.
	\end{align*}
Invoking \cref{def:err_D} again, we obtain:	
		\begin{align}\label{eqn:erc_D_elementary_ineq1}
	 	\mathcal{D}\left(x_{k+1}, y\right) & \leq  \mathcal{D}\left(x_{k}, y\right) - 2\,\prob{{i_k}}^{-1}\gamma_k\left( x^{(i_k)}_{k} - y^{(i_k)}\right)^TG_{k,i_k}\left(x_k\right)+ \prob{{i_k}}^{-1}\gamma_k^2 \left\Vert\left.G_{k,i_k}\left(x_k\right)\right\Vert\right.^2.
	\end{align}
From \cref{def:deltas} and \cref{def:reg_map}, we can write:
\begin{align*}
&\prob{{i_k}}^{-1}\left( x^{(i_k)}_{k} - y^{(i_k)}\right)^TG_{k,i_k}\left(x_k\right) = \prob{{i_k}}^{-1}(x_k-y)^T\left({\mathbf{U}}_{i_k}G_{k,i_k}\left(x_k\right)\right) \\ 
&=  \prob{{i_k}}^{-1}(x_k-y)^T\left({\mathbf{U}}_{i_k}F_{i_k}(x_k)+\eta_k{\mathbf{U}}_{i_k}\tilde \nabla_{i_k}f\left(x_k\right)\right) \\
& {=(x_k-y)^T\left(F(x_k)-\Delta_k+\eta_k\tilde \nabla f(x_k)-\eta_k\delta_k\right) = (x_k-y)^T\left(G_k(x_k)-\Delta_k-\eta_k\delta_k\right).}
\end{align*}
	Combining the preceding inequality and relation \cref{eqn:erc_D_elementary_ineq1}, we obtain:
	\begin{align}\label{eqn:erc_D_elementary_ineq2}
	 	\mathcal{D}\left(x_{k+1}, y\right) & \leq  \mathcal{D}\left(x_{k}, y\right) {- 2 \gamma_k(x_k-y)^T\left(G_k(x_k)-\Delta_k-\eta_k\delta_k\right)}+ \prob{{i_k}}^{-1}\gamma_k^2 \left\Vert\left.G_{k,i_k}\left(x_k\right)\right\Vert\right.^2.
	\end{align}
	Consider the definition of the auxiliary sequence $\{u_k\}$ in \cref{lem:prebound_gap}. Invoking the nonexpansiveness property of the projection again, we can obtain:
\begin{align*}
\|u_{k+1}-y\|^2  &\leq{\|u_k-\gamma_k\left(\Delta_k+\eta_k\delta_k\right)-y\|^2}\\
&=\|u_k-y\|^2{-2\gamma_k(u_k-y)^T\left(\Delta_k+\eta_k\delta_k\right)} + \gamma_k^2\|\Delta_k+\eta_k\delta_k\|^2.
\end{align*}
Thus, we have: 
\begin{align*}
\|u_{k+1}-y\|^2\leq \|u_k-y\|^2{- 2\gamma_k(u_k-y)^T\left(\Delta_k+\eta_k\delta_k\right)} +2\gamma_k^2\|\Delta_k\|^2+2\gamma_k^2\eta_k^2\|\delta_k\|^2.
\end{align*}
Adding the preceding inequality and the inequality \cref{eqn:erc_D_elementary_ineq2} together, we obtain:
\begin{align}\label{eqn:erc_D_elementary_ineq3}
&2\gamma_k(x_k-y)^T G_k(x_k) \leq \left(\mathcal{D}\left(x_{k}, y\right) +\|u_k-y\|^2\right) -\left(\mathcal{D}\left(x_{k+1}, y\right) +\|u_{k+1}-y\|^2\right)\nonumber\\
& {+ 2\gamma_k(x_k-u_k)^T\left(\Delta_k+\eta_k\delta_k\right)}+2\gamma_k^2\left(\|\Delta_k\|^2+\eta_k^2\|\delta_k\|^2\right)+ \prob{{i_k}}^{-1}\gamma_k^2 \left\Vert\left.G_{k,i_k}\left(x_k\right)\right\Vert\right.^2.
\end{align}
From the monotonicity property of the mapping $F$ and \cref{def:reg_map}, we have:
\begin{align*}
(x_k-y)^TG_k(x_k)
\geq (x_k-y)^TF(y)+\eta_k\tilde \nabla f(x_k)^T(x_k-y).
\end{align*}
This provides a lower bound on the left-hand side of \cref{eqn:erc_D_elementary_ineq3}. The inequality \cref{eqn:prebound_gap} is obtained by substituting this bound in \cref{eqn:erc_D_elementary_ineq3} and multiplying both sides by $\frac{\gamma_{k}^{r-1}}{2}$.
\end{proof}

In the following, we develop upper bounds for suboptimality and infeasibility of the weighted average iterate generated by \cref{alg:aRB-IRG}. Both of these error bounds are characterized in terms of the stepsize and the regularization parameter. 
\begin{proposition}[Error bounds for \cref{alg:aRB-IRG}]\label{prop:dual_gap_bound}
Let the sequence $\{\bar x_k\}$ be generated by \cref{alg:aRB-IRG}{, where $0\leq r<1$}. Suppose $\{\gamma_k\}$ and $\{\eta_k\}$ are strictly positive and nonincreasing sequences. Let \cref{assum:problem} and \cref{assum:random sample} hold {and} assume {that} the set $X$ is bounded, i.e., {$\|x\|\leq M$} for all $x \in X$ and some $M>0$. \\
\noindent (a) Let $x^*$ be an optimal solution to problem \cref{prob:main}. Then, for {all} $N\geq 1$:
\begin{align}\label{eqn:f_bound} 
&\EXP{f\left(\bar x_N\right)}-f(x^*) \leq \frac{\frac{4M^2\gamma_{{N}}^{r-1}}{\eta_{N}} +\sum_{k=0}^{N}\eta_k^{-1}\gamma_k^{r+1}\left(C_F^2+\eta_k^2C_f^2\right)}{\prob{{min}}\sum_{k=0}^{N}\gamma_k^r}.
\end{align}
\noindent (b) Consider the dual gap function in \cref{def:gap}. Then, for {all} $N\geq 1$:
\begin{align}\label{eqn:gap_bound} 
\EXP{\mathrm{GAP}\left(\bar{x}_N\right)} \leq\frac{ {4 M^2 \gamma_{{N}}^{r-1} }  +\sum_{k=0}^{N}\gamma_{{k}}^r\left( 2\prob{{min}}\eta_kC_fM+\gamma_k C_F^2 +  \gamma_{k}\eta_k^2C_f ^2 \right)}{\prob{{min}}\sum_{k=0}^{N} \gamma_{{k}}^r}.
\end{align}
\end{proposition}

\begin{proof}
We define the following terms for all $k\geq 0$, that appear in \cref{eqn:prebound_gap}:
\begin{align}\label{eqn:three_terms_abr}
& {\Theta_{k,1} \triangleq \gamma_k^r(x_k-u_k)^T\left(\Delta_k+\eta_k\delta_k\right)}, \quad \Theta_{k,2} \triangleq \gamma_k^{r+1}\left(\|\Delta_k\|^2+\eta_k^2\|\delta_k\|^2\right),\nonumber \\ 
&\Theta_{k,3} \triangleq 0.5\prob{{i_k}}^{-1}\gamma_k^{1+r}\left\Vert\left.G_{k,i_k}\left(x_k\right)\right\Vert\right.^2.
\end{align}
Next, we estimate the expected values of these terms. {Consider the notation of} $\mathcal{F}_k$ given by \cref{def:Fk}. Note that $x_k$ is $\mathcal{F}_k$--measurable. Also, from the definition of $u_k$ in \cref{lem:prebound_gap}, $u_k$ is $\mathcal{F}_k$--measurable. Note, however, that $\Theta_{k,j}$ is $\mathcal{F}_{k+1}$--measurable for all $j \in \{1,2,3\}$. Taking these into account and using the total probability law, for any $k\geq 0$ and $j \in \{1,2,3\}$, we have $ \EXP{\Theta_{k,j}} = \mathsf{E}_{\mathcal{F}_k}\left[\mathsf{E}_{i_k}\left[\Theta_{k,j}\mid \mathcal{F}_k\right]\right]$. From this relation and \cref{Lemma 2.12}, we have for any $k\geq 0$:
\begin{align}\label{eqn:estimate_deltas_1_2}
&\EXP{\Theta_{k,1}} = 0, \qquad \EXP{\Theta_{k,2}} = \left(\prob{{min}}^{-1}-1\right)\gamma_k^{r+1}\left(C_F^2+\eta_k^2C_f^2\right).
\end{align}
Also, using \cref{def:reg_map} and the triangle inequality, we can write:
\begin{align*}
&{\mathsf{E}_{i_k}\left[\Theta_{k,3}\mid \mathcal{F}_k\right]} = \sum_{i=1}^d\prob{i}\left(0.5\prob{{i}}^{-1}\gamma_k^{1+r}\left\Vert\left.G_{k,i}\left(x_k\right)\right\Vert\right.^2\right)\nonumber\\
&\leq \gamma_k^{1+r}\sum_{i=1}^d\left(\|F_i(x_k)\|^2+\eta_k^2\|\tilde \nabla_i f(x_k)\|^2\right)  = \gamma_k^{1+r}\|F(x_k)\|^2+\eta_k^2\|\tilde \nabla f(x_k)\|^2.
\end{align*}
{From the preceding inequality, we obtain:
\begin{align}\label{eqn:estimate_deltas_3}
\EXP{\Theta_{k,3}} & \leq  \gamma_k^{1+r}\left(C_F^2+\eta_k^2C_f^2\right).
\end{align}
}We are now ready to show the inequalities \cref{eqn:f_bound} and \cref{eqn:gap_bound} as follows:\\
\noindent (a) Consider \cref{eqn:prebound_gap}. From the definition of subgradients of the convex function $f$, we have that $f(x_k)-f(y) \leq \tilde \nabla f(x_k)^T(x_k-y)$. Thus, {from \cref{eqn:three_terms_abr}} we obtain for any $y \in X$:
\begin{align*}
&\gamma_k^rF(y)^T(x_k-y) +\gamma_k^r\eta_k\left(f(x_k)-f(y)\right)\leq \frac{\gamma_k^{r-1}}{2}\left(\mathcal{D}\left(x_{k}, y\right) +\|u_k-y\|^2\right) \nonumber\\
&-\frac{\gamma_k^{r-1}}{2}\left(\mathcal{D}\left(x_{k+1}, y\right) +\|u_{k+1}-y\|^2\right)+ \Theta_{k,1}+\Theta_{k,2}+\Theta_{k,3}.
\end{align*}
Let us substitute $y:=x^*$, where $x^*$ denotes an optimal solution to \cref{prob:main}. Note that $x^*$ must be a feasible solution to \cref{prob:main}, i.e., $F(x^*)^T(x_k-x^*) \geq 0$. Thus, we obtain:
\begin{align}\label{eqn:gap_prop_ineq1_a}
			&\gamma_k^r\eta_k\left(f(x_k)-f(x^*)\right) \leq  \frac{\gamma_k^{r-1}}{2}\left(\mathcal{D}\left(x_{k}, x^*\right) +\|u_k-x^*\|^2\right) \nonumber\\
&-\frac{\gamma_k^{r-1}}{2}\left(\mathcal{D}\left(x_{k+1}, x^*\right) +\|u_{k+1}-x^*\|^2\right)+\Theta_{k,1}+\Theta_{k,2}+\Theta_{k,3}.
\end{align}
Dividing both sides by $\eta_k$ and adding and subtracting $\frac{\gamma_{{k-1}}^{r-1}}{2\eta_{k-1}} \left( \mathcal{D}\left(x_k,x^*\right) + \left\Vert u_k-x^* \right\Vert^2 \right)$ in the right-hand side of \cref{eqn:gap_prop_ineq1_a}, we obtain for $k\geq 1$:
\begin{align}\label{eqn:gap_prop_ineq2_a}
&\gamma_k^r\left(f(x_k)-f(x^*)\right) \leq  \frac{\gamma_{{k-1}}^{r-1}}{2\eta_{k-1}} \left(  \mathcal{D}\left(x_k,x^*\right)  + \left\Vert u_k-x^* \right\Vert^2 \right) \nonumber \\
&- \frac{\gamma_{{k}}^{r-1}}{2\eta_k} \left(  \mathcal{D}\left(x_{k+1},x^*\right)  + \left\Vert u_{k+1}-x^* \right\Vert^2 \right)\nonumber \\ 
&+ \frac{1}{2} \left(\frac{\gamma_{{k}}^{r-1}}{\eta_k} - \frac{\gamma_{{k-1}}^{r-1}}{\eta_{k-1}}\right)\left( \mathcal{D}\left(x_k,x^*\right) + \left\Vert u_k-x^* \right\Vert^2 \right)+\eta_k^{-1}\left( \Theta_{k,1}+\Theta_{k,2}+\Theta_{k,3}\right).
\end{align}
Since $r-1 < 0$ and that $\{\gamma_k\}$ and $\{\eta_k\}$ are nonincreasing, we have $\frac{\gamma_{{k}}^{r-1}}{\eta_k} - \frac{\gamma_{{k-1}}^{r-1}}{\eta_{k-1}} \geq 0$. Also, from the boundedness of the set $X$, since $x_k$, $x^*$, and {$u_k$} belong to $ X$, using \cref{rem:err_D} and the triangle inequality, we have:
\begin{align}\label{eqn:gap_prop_ineq3_a}
\mathcal{D}\left(x_k,x^*\right)  + \left\Vert u_k-x^* \right\Vert^2 
\leq \prob{{min}}^{-1}\left\Vert x_k-x^* \right\Vert^2 + \left\Vert u_k-x^* \right\Vert^2
\leq 4M^2\left(\prob{{min}}^{-1}+1\right)\leq \frac{8M^2}{\prob{{min}}}.
\end{align}
Summing over \cref{eqn:gap_prop_ineq2_a} from $k=1$ to $N$ and using \cref{eqn:gap_prop_ineq3_a}, we obtain:
\begin{align*}
&\sum_{k=1}^{N}\gamma_k^r\left(f(x_k)-f(x^*)\right) \leq \frac{\gamma_{0}^{r-1}}{2\eta_0} \left(  \mathcal{D}\left(x_1,x^*\right)  + \left\Vert u_1-x^* \right\Vert^2 \right) 
\nonumber \\ 
&+4M^2\prob{{min}}^{-1}\left(\frac{\gamma_{{N}}^{r-1}}{\eta_{N}} - \frac{\gamma_{{0}}^{r-1}}{\eta_{0}}\right)+\sum_{k=1}^{N}\eta_k^{-1}\left(\Theta_{k,1}+\Theta_{k,2}+\Theta_{k,3}\right),
\end{align*}
where we drop the nonpositive term. From relation \cref{eqn:gap_prop_ineq1_a} when $k=0$, we have:
\begin{align*}
\gamma_0^r\left(f(x_0)-f(x^*)\right)& \leq  \frac{\gamma_0^{r-1}}{2\eta_0}\left(\mathcal{D}\left(x_{0}, x^*\right) +\|u_0-x^*\|^2\right) \nonumber\\
&-\frac{\gamma_0^{r-1}}{2\eta_0}\left(\mathcal{D}\left(x_{1}, x^*\right) +\|u_{1}-x^*\|^2\right)+ \eta_0^{-1}\left(\Theta_{0,1}+\Theta_{0,2}+\Theta_{0,3}\right).
\end{align*}
Adding the last two inequalities, {multiplying and dividing the left-hand side by $\sum_{k=0}^{N}\gamma_k^r$, and then, invoking} \cref{Lemma 2.10} and convexity of $f$, we obtain:
\begin{align*}
&\left(\sum_{k=0}^{N}\gamma_k^r\right)\left(f\left(\bar x_N\right)-f(x^*)\right) \leq \frac{\gamma_{0}^{r-1}}{2\eta_0} \left(  \mathcal{D}\left(x_0,x^*\right)  + \left\Vert u_0-x^* \right\Vert^2 \right) 
\nonumber \\ 
&+4M^2\prob{{min}}^{-1}\left(\frac{\gamma_{{N}}^{r-1}}{\eta_{N}} - \frac{\gamma_{{0}}^{r-1}}{\eta_{0}}\right)+\sum_{k=0}^{N}\eta_k^{-1}\left( \Theta_{k,1}+\Theta_{k,2}+\Theta_{k,3}\right).
\end{align*}
Taking the expectation on both sides and invoking \cref{eqn:gap_prop_ineq3_a}, we obtain:
\begin{align*}
&\EXP{f\left(\bar x_N\right)}-f(x^*) \leq \frac{\frac{4M^2\prob{{min}}^{-1}\gamma_{{N}}^{r-1}}{\eta_{N}} +\sum_{k=0}^{N}\eta_k^{-1}\EXP{\Theta_{k,1}+\Theta_{k,2}+\Theta_{k,3}}}{\sum_{k=0}^{N}\gamma_k^r}.
\end{align*}
From the relations \cref{eqn:estimate_deltas_1_2} and \cref{eqn:estimate_deltas_3}, we obtain:
\begin{align*}
&\EXP{f\left(\bar x_N\right)}-f(x^*) \leq \frac{\frac{4M^2\prob{{min}}^{-1}\gamma_{{N}}^{r-1}}{\eta_{N}}+\sum_{k=0}^{N}\eta_k^{-1}\left(\prob{{min}}^{-1}\gamma_k^{r+1}\left(C_F^2+\eta_k^2C_f^2\right)\right)}{\sum_{k=0}^{N}\gamma_k^r},
\end{align*}
which implies the inequality \cref{eqn:f_bound}. 

\noindent (b) From the Cauchy-Schwarz inequality, the definitions of $C_f$ and $M$, and the triangle inequality, we have $ \tilde \nabla f(x_k)^T(y-x_k) \leq \left\| \tilde \nabla f(x_k)\right\|\|x_k-y\| \leq 2C_f M.$
Adding the preceding inequality with the relation \cref{eqn:prebound_gap}, {from \cref{eqn:three_terms_abr}} we obtain:
\begin{align}\label{eqn:prop_gap_proof_ineq0}
&\gamma_k^rF(y)^T(x_k-y) \leq \frac{\gamma_k^{r-1}}{2}\left(\mathcal{D}\left(x_{k}, y\right) +\|u_k-y\|^2\right) \nonumber\\
&-\frac{\gamma_k^{r-1}}{2}\left(\mathcal{D}\left(x_{k+1}, y\right) +\|u_{k+1}-y\|^2\right)+ 2\gamma_k^r\eta_kC_f M+ \Theta_{k,1}+\Theta_{k,2}+\Theta_{k,3}.
\end{align}
Adding and subtracting the term $\frac{\gamma_{{k-1}}^{r-1}}{2} \left( \mathcal{D}\left(x_k,y\right) + \left\Vert u_k-y \right\Vert^2 \right)$, we obtain:
\begin{align*}
			&\gamma_k^rF(y)^T(x_k-y) \leq  \frac{\gamma_{{k-1}}^{r-1}}{2} \left(  \mathcal{D}\left(x_k,y\right)  + \left\Vert u_k-y \right\Vert^2 \right) \nonumber \\
			&- \frac{\gamma_{{k}}^{r-1}}{2} \left(  \mathcal{D}\left(x_{k+1},y\right)  + \left\Vert u_{k+1}-y \right\Vert^2 \right)+ \frac{1}{2} \left(\gamma_{{k}}^{r-1} - \gamma_{{k-1}}^{r-1}\right)\left( \mathcal{D}\left(x_k,y\right) + \left\Vert u_k-y \right\Vert^2 \right)\nonumber \\
&+ 2\gamma_k^r\eta_kC_f M+ \Theta_{k,1}+\Theta_{k,2}+\Theta_{k,3}.
\end{align*}
Substituting the bound given by \cref{eqn:gap_prop_ineq3_a} in the preceding relation, we obtain:
\begin{align*}
			&\gamma_k^rF(y)^T(x_k-y) \leq  \frac{\gamma_{{k-1}}^{r-1}}{2} \left(  \mathcal{D}\left(x_k,y\right)  + \left\Vert u_k-y \right\Vert^2 \right) \nonumber \\
			&- \frac{\gamma_{{k}}^{r-1}}{2} \left(  \mathcal{D}\left(x_{k+1},y\right)  + \left\Vert u_{k+1}-y \right\Vert^2 \right)+ 4M^2\prob{{min}}^{-1} \left(\gamma_{{k}}^{r-1} - \gamma_{{k-1}}^{r-1}\right)\nonumber \\
&+ 2\gamma_k^r\eta_kC_f M+ \Theta_{k,1}+\Theta_{k,2}+\Theta_{k,3}.
\end{align*}
Summing both sides from $k = 1 $ to $N$, we obtain:
\begin{align}\label{eqn:prop_gap_proof_ineq1}
\sum_{k=1}^{N}\gamma_{{k}}^r F(y)^T \left(x_k-y\right)& \leq  \frac{\gamma_{{0}}^{r-1}}{2} \left( \mathcal{D}\left(x_1,y\right) + \left\Vert u_1-y \right\Vert^2 \right)+ 4 M^2\prob{{min}}^{-1} \left( \gamma_{{N}}^{r-1} - \gamma_{{0}}^{r-1}\right) \nonumber \\
			& +\sum_{k=1}^{N}\left(  2\gamma_k^r\eta_kC_f M+ \Theta_{k,1}+\Theta_{k,2}+\Theta_{k,3}\right).
\end{align}
Writing the inequality \cref{eqn:prop_gap_proof_ineq0} for $k =0$, we have:
	\begin{align}\label{eqn:prop_gap_proof_ineq2}
\gamma_0^rF(y)^T(x_0-y) &\leq \frac{\gamma_0^{r-1}}{2}\left(\mathcal{D}\left(x_{0}, y\right) +\|u_0-y\|^2\right)-\frac{\gamma_0^{r-1}}{2}\left(\mathcal{D}\left(x_{1}, y\right) +\|u_{1}-y\|^2\right) \nonumber\\
&+ 2\gamma_0^r\eta_0C_f M+ \Theta_{0,1}+\Theta_{0,2}+\Theta_{0,3}.
\end{align}
Adding \cref{eqn:prop_gap_proof_ineq1} and \cref{eqn:prop_gap_proof_ineq2} together, we obtain:
\begin{align}\label{eqn:prop_gap_proof_ineq3}
\sum_{k=0}^{N}\gamma_{{k}}^r F(y)^T \left(x_k-y\right) &\leq  \frac{\gamma_{{0}}^{r-1}}{2} \left( \mathcal{D}\left(x_0,y\right) + \left\Vert u_0 -y \right\Vert^2 \right) + 4 M^2\prob{{min}}^{-1} \left( \gamma_{{N}}^{r-1} - \gamma_{{0}}^{r-1}\right) \nonumber \\& +\sum_{k=0}^{N}\left(  2\gamma_k^r\eta_kC_f M+ \Theta_{k,1}+\Theta_{k,2}+\Theta_{k,3}\right).
\end{align}
Recalling $\bar{x}_{N} = \sum_{k=0}^N \lambda_{k,N} x_k$ in \cref{Lemma 2.10}, applying the bound given by \cref{eqn:gap_prop_ineq3_a}, and using the triangle inequality, we obtain:
	\begin{align*}
\left(\sum_{k=0}^{N} \gamma_{{k}}^r \right) F(y)^T \left(\bar{x}_N-y\right) \leq& {4 M^2\prob{{min}}^{-1}  \gamma_{{N}}^{r-1} }  +\sum_{k=0}^{N}\left(  2\gamma_k^r\eta_kC_f M+ \Theta_{k,1}+\Theta_{k,2}+\Theta_{k,3}\right).
\end{align*}
Taking the supremum with respect to $y$ over the set $X$ from the left-hand side, invoking \cref{def:gap}, and then dividing both sides by $\sum_{k=0}^{N} \gamma_{{k}}^r $, we obtain:
	\begin{align*}
\text{GAP}\left(\bar x_N\right) \leq \frac{4 M^2\prob{{min}}^{-1}  \gamma_{{N}}^{r-1}   +\sum_{k=0}^{N}\left(  2\gamma_k^r\eta_kC_f M+ \Theta_{k,1}+\Theta_{k,2}+\Theta_{k,3}\right)}{\sum_{k=0}^{N} \gamma_{{k}}^r }.
\end{align*}
Taking the expectation on both sides, using the relations \cref{eqn:estimate_deltas_1_2} and \cref{eqn:estimate_deltas_3}, and rearranging the terms, we obtain the inequality \cref{eqn:gap_bound}.
\end{proof}
We are now ready to present the convergence rate results of the proposed method. 
\begin{theorem}[Convergence rate statements for \cref{alg:aRB-IRG}]\label{thm:a-IRG_rate_results}
Consider \cref{alg:aRB-IRG}. Let \cref{assum:problem} and \cref{assum:random sample} hold and assume {that} the set $X$ is bounded such that $\|x\|\leq M$ for all $x \in X$ and some $M>0$. Suppose for {all} $k\geq 0$, $\gamma_k:=\frac{\gamma_0}{\sqrt{k+1}}$ and $\eta_k:=\frac{\eta_0}{(k+1)^b}$, where $\gamma_0>0$, $\eta_0>0$, and $0<b<0.5$. Then, for any $0\leq r <1$, the following results hold:\\
\noindent (i) Let $x^*$ be an optimal solution to the problem \cref{prob:main}. Then, for {all} $N\geq 2^{\frac{2}{1-r}}-1$:
\begin{align}\label{eqn:f_rate} 
&\EXP{f\left(\bar x_N\right)}-f(x^*) \leq \frac{2-r}{\prob{{min}}\eta_0}\left(\frac{4M^2}{\gamma_0}+\frac{\gamma_0\left(C_F^2+\eta_0^2C_f^2\right)}{0.5-0.5r+b}\right)\frac{1}{(N+1)^{0.5-b}}.
\end{align}
\noindent (ii) Consider the dual gap function in \cref{def:gap}. Then, for {all} $N\geq 2^{\frac{2}{1-r}}-1$:
\begin{align}\label{eqn:gap_rate} 
\EXP{\mathrm{GAP}\left(\bar{x}_N\right)} \leq \frac{2-r}{\prob{{min}}}\left(\frac{4M^2}{\gamma_0}+\frac{\gamma_0\left(C_F^2+\eta_0^2C_f^2\right)}{0.5-0.5r}+\frac{2\prob{{min}C_fM}\eta_0}{1-0.5r-b}\right)\frac{1}{(N+1)^b}.
\end{align}
\end{theorem}
\begin{proof} Let us define the following terms: 
\begin{align*}
& \Lambda_{N,1} \triangleq \prob{{min}}\sum_{k=0}^{N}\gamma_k^r, \quad  \Lambda_{N,2}\triangleq  \frac{4M^2\gamma_{{N}}^{r-1}}{\eta_{N}} ,\quad \Lambda_{N,3} \triangleq \left(C_F^2+\eta_0^2C_f^2\right)\sum_{k=0}^{N}\eta_k^{-1}\gamma_k^{r+1},\nonumber\\ 
& \Lambda_{N,4} \triangleq 4M^2\gamma_{{N}}^{r-1}, \quad  \Lambda_{N,5} \triangleq \left(C_F^2+\eta_0^2C_f^2\right)\sum_{k=0}^{N}\gamma_k^{r+1}, \quad \Lambda_{N,6} \triangleq 2\prob{{min}}C_fM \sum_{k=0}^{N}\eta_k\gamma_k^r.
\end{align*}
Note that from \cref{eqn:f_bound} and \cref{eqn:gap_bound}, we have:
\begin{align}\label{eqn:bounds_in_Lambdas}
\EXP{f\left(\bar x_N\right)}-f(x^*) \leq \frac{\Lambda_{N,2}+\Lambda_{N,3}}{\Lambda_{N,1}}, \quad \EXP{\text{GAP}\left(\bar{x}_N\right)} \leq \frac{\Lambda_{N,4}+\Lambda_{N,5}+\Lambda_{N,6}}{\Lambda_{N,1}}.
\end{align}
Next, we apply \cref{Lemma 2.13} to estimate the terms $\Lambda_{N,i}$. Substituting $\gamma_k$ and $\eta_k$ by their update rules, we obtain:
\begin{align*}
\Lambda_{N,1} &=\prob{{min}}\sum_{k=0}^{N}\frac{\gamma_0^r}{(k+1)^{0.5r}}\geq \frac{\prob{{min}}\gamma_0^r(N+1)^{1-0.5r}}{2(1-0.5r)},\\  
\Lambda_{N,2} &= \frac{4M^2(N+1)^{0.5(1-r)+b}}{\eta_0\gamma_0^{1-r}},\quad \Lambda_{N,4} = \frac{4M^2(N+1)^{0.5(1-r)}}{\gamma_0^{1-r}}, \\
\Lambda_{N,3} &= \sum_{k=0}^{N}\frac{\left(C_F^2+\eta_0^2C_f^2\right)\gamma_0^{1+r}}{\eta_0(k+1)^{0.5(1+r)-b}} \leq\frac{\gamma_0^{1+r}\left(C_F^2+\eta_0^2C_f^2\right)(N+1)^{1-0.5(1+r)+b}}{\eta_0(1-0.5(1+r)+b)},\\
\Lambda_{N,5} &= \left(C_F^2+\eta_0^2C_f^2\right)\sum_{k=0}^{N}\frac{\gamma_0^{r+1}}{(k+1)^{0.5(1+r)}}\leq \frac{\left(C_F^2+\eta_0^2C_f^2\right)\gamma_0^{r+1}(N+1)^{1-0.5(1+r)}}{1-0.5(1+r)},\\
\Lambda_{N,6} &=2\prob{{min}}C_fM \eta_0\gamma_0^r\sum_{k=0}^{N}\frac{1}{(k+1)^{0.5r+b}}\leq  \frac{2\prob{{min}}C_fM \eta_0\gamma_0^r (N+1)^{1-0.5r-b}}{1-0.5r-b}.
\end{align*}
For these inequalities to hold, we need to ensure that {the} conditions of \cref{Lemma 2.13} are met. Accordingly, we must have $0\leq 0.5r <1$, $0\leq 0.5(1+r)-b <1$, $0\leq 0.5r+b<1$, and $0 \leq 0.5(1+r)<1$. These relations hold because $0\leq r<1$ and $0<b<0.5$. Another set of conditions when applying \cref{Lemma 2.13} includes $N\geq \max\left\{2^{1/(1-0.5r)},2^{1/(1-0.5(1+r)+b)},2^{1/(1-0.5r-b)},2^{1/(1-0.5(1+r))}\right\}-1$. {This relation is indeed} satisfied as a consequence of $N\geq 2^{\frac{2}{1-r}}-1$, {$0<b<0.5$}, and $0\leq r<1$. We conclude that all the necessary conditions for applying \cref{Lemma 2.13} and obtaining the aforementioned bounds for the terms $\Lambda_{N,i}$ are satisfied. To show that the inequalities \cref{eqn:f_rate} and \cref{eqn:gap_rate} hold, it suffices to substitute the preceding {bounds on} the terms $\Lambda_{N,i}$ into the two inequalities given by \cref{eqn:bounds_in_Lambdas}. The details are as follows:
\begin{align*}
&\EXP{f\left(\bar x_N\right)}-f(x^*) \leq \frac{\Lambda_{N,2}+\Lambda_{N,3}}{\Lambda_{N,1}}=\frac{2-r}{\prob{{min}}\gamma_0^r(N+1)^{1-0.5r}}\left(\frac{4M^2(N+1)^{0.5-0.5r+b}}{\eta_0\gamma_0^{1-r}}\right.\\
&\left.+\left(\frac{\gamma_0^{1+r}}{\eta_0}\right)\frac{\left(C_F^2+\eta_0^2C_f^2\right)(N+1)^{0.5-0.5r+b}}{0.5-0.5r+b}\right). 
\end{align*}
The inequality \cref{eqn:f_rate} is obtained by rearranging the terms in the preceding relation.
\begin{align*}
&\EXP{\text{GAP}\left(\bar{x}_N\right)} \leq \frac{\Lambda_{N,4}+\Lambda_{N,5}+\Lambda_{N,6}}{\Lambda_{N,1}} \leq\frac{2-r}{\prob{{min}}\gamma_0^r(N+1)^{1-0.5r}}\left(\frac{4M^2(N+1)^{0.5-0.5r}}{\gamma_0^{1-r}} \right.\\
&\left.+\frac{\left(C_F^2+\eta_0^2C_f^2\right)\gamma_0^{r+1}(N+1)^{0.5-0.5r}}{0.5-0.5r}+ \frac{2\prob{{min}}C_fM \eta_0\gamma_0^r (N+1)^{1-0.5r-b}}{1-0.5r-b}\right).
\end{align*}
Then, \cref{eqn:gap_rate} can be obtained by rearranging the terms in the preceding inequality.
\end{proof} 
\begin{remark}[Iteration complexity of \cref{alg:aRB-IRG}]\label{rem:iter_complexity}
As an immediate result from \cref{thm:a-IRG_rate_results}, choosing $\gamma_k:=\frac{\gamma_0}{\sqrt{k+1}}$ and $\eta_k:=\frac{\eta_0}{\sqrt[4]{k+1}}$, we obtain: $$\EXP{f\left(\bar x_N\right)-f(x^*)} =\EXP{\text{GAP}\left(\bar{x}_N\right)} = {\mathcal{O}\left(\frac{1}{\sqrt[4]{N}}\right)}.$$ This implies that \cref{alg:aRB-IRG} achieves an iteration complexity of {$\mathcal{O}\left(\epsilon^{-4}\right)$ in solving \cref{prob:main}, where $\epsilon>0$ denotes the expected tolerance in both of the suboptimality and infeasibility metrics}.
\end{remark}
{The rate statements derived in \cref{thm:a-IRG_rate_results} are in a mean sense. In the following, we consider a deterministic variant of \cref{alg:aRB-IRG} where we suppress the randomized block-coordinate scheme. The outline of this deterministic method is presented by \cref{alg:a-IRG}. In \cref{cor:a-IRG_rate_results}, we show that non-asymptotic deterministic rate statements can be derived for \cref{alg:a-IRG}.   
\begin{algorithm} [h]
	\caption{a-IRG}
	\begin{algorithmic}[1]
		\STATE {\textbf{Input:} An arbitrary initial point $x_0 \in X$, $\bar{x}_0 := x_0$, initial stepsize $\gamma_{0} > 0$, initial regularization parameter $\eta_{0} > 0$, a scalar $0\leq r<1$, and $S_0 := \gamma_0^r$.}
		\FOR {{k = 0, 1, \dots}}
		\STATE {Evaluate $F(x_k)$ and $\tilde \nabla f(x_k)$ where $\tilde \nabla f(x_k) \in \partial f(x_k)$.}

		\STATE {For all $i \in \{1,\ldots,d\}$, do the following updates:
		\begin{align}\label{equ:update_rule_aIRG} 
		{x_{k+1}^{(i)}} :=
		\proj[X_{i}]{ x_{k}^{(i)}-\gamma_{k}\left(F_{i}\left( x_{k}\right) + \eta_{k} \tilde \nabla_{{i}} f\left(x_{k}\right)\right)}.
		\end{align}}
		\STATE {Obtain $\gamma_{k+1}$ and $\eta_{k+1}$ (cf. \cref{cor:a-IRG_rate_results} for the update rules).}

		\STATE {Update the averaged iterate $\bar{x}_{k}$ as follows:
			\begin{align}\label{eqn:ave_step_of_alg_v2}
			S_{k+1} &:= S_k + \gamma_{k+1}^r,\quad	\bar{x}_{k+1}:= \frac{S_k\bar{x}_{k}+\gamma_{k+1}^rx_{k+1}}{S_{k+1}}.
			\end{align}}
		\ENDFOR
	\end{algorithmic} \label{alg:a-IRG}
\end{algorithm}
\begin{corollary}[Convergence rate statements for \cref{alg:a-IRG}]\label{cor:a-IRG_rate_results}
Consider \cref{alg:a-IRG}. Let \cref{assum:problem} hold and assume that the set $X$ is bounded such that $\|x\|\leq M$ for all $x \in X$ and some $M>0$. Suppose for $k\geq 0$, $\gamma_k:=\frac{\gamma_0}{\sqrt{k+1}}$ and $\eta_k:=\frac{\eta_0}{(k+1)^b}$, where $\gamma_0>0$, $\eta_0>0$, and $0<b<0.5$. Then, for any $0\leq r <1$, the following results hold:\\
\noindent (i) Let $x^*$ be an optimal solution to the problem \cref{prob:main}. Then, for all $N\geq 2^{\frac{2}{1-r}}-1$:
\begin{align}\label{eqn:f_rate_deter} 
&{f\left(\bar x_N\right)}-f(x^*) \leq \frac{2-r}{\eta_0}\left(\frac{4M^2}{\gamma_0}+\frac{\gamma_0\left(C_F^2+\eta_0^2C_f^2\right)}{0.5-0.5r+b}\right)\frac{1}{(N+1)^{0.5-b}}.
\end{align}
\noindent (ii) Consider the dual gap function in \cref{def:gap}. Then, for all $N\geq 2^{\frac{2}{1-r}}-1$:
\begin{align}\label{eqn:gap_rate_deter} 
{\mathrm{GAP}\left(\bar{x}_N\right)} \leq (2-r)\left(\frac{4M^2}{\gamma_0}+\frac{\gamma_0\left(C_F^2+\eta_0^2C_f^2\right)}{0.5-0.5r}+\frac{2C_fM\eta_0}{1-0.5r-b}\right)\frac{1}{(N+1)^b}.
\end{align}
\end{corollary}
\begin{proof}
See \cref{app:a-IRG_rate_results}.
\end{proof}
}
\section{Addressing the case where $X$ is unbounded}\label{sec:conv_unbounded}
The convergence and rate statements provided by \cref{thm:a-IRG_rate_results} require the set $X$ to be bounded. We, however, note that in some applications, e.g., in the models presented in \cref{ex:opt_compl_const} and \cref{ex:opt_nl_const}, this assumption may not hold. Accordingly, in this section, our aim is to analyze the convergence of \cref{alg:aRB-IRG} when $X$ is unbounded. To this end, we consider the following main assumption:
\begin{assumption}\label{assum:problem_unboundedX} Consider problem \cref{prob:main} under the following conditions:

\noindent (a) The set $X_i$ is nonempty, closed, and convex for all $i=1,\dots, d$.

\noindent (b) The function $f$ is continuously differentiable and $\mu_f$--strongly convex over $X$.

\noindent  (c) The mapping $F:\mathbb{R}^n \to \mathbb{R}^n$ is continuous and monotone over $X$. 

\noindent (d) The solution set $\mathrm{SOL}(X,F)$ {is} nonempty.
\end{assumption}
\begin{remark}[Existence and uniqueness of the optimal solution]\label{rem:unique_x_star}
	Under \cref{assum:problem_unboundedX}, the constraint set of \cref{prob:main}, i.e., $\text{SOL}(X,F)$, is nonempty, closed, and convex. The convexity of this set is implied by Theorem 2.3.5 in \cite{FacchineiPang2003} and its closedness property is obtained by the continuity of the mapping $F$ and closedness of the set $X$. Because in the problem \cref{prob:main}, the objective function $f$ is strongly convex and that the constraint set is nonempty, closed, and convex, {we conclude from Proposition 1.1.2 in \cite{BertsekasNLPBook2016}} that the problem \cref{prob:main} has a unique optimal solution. Throughout this section, we let $x^*$ denote this unique optimal solution. 
\end{remark}
\subsection{Preliminaries}
In this part, we provide some preliminary results that will be used in the convergence analysis. We begin by defining a generalized {variant of the} Tikhonov trajectory that is associated with the problem of interest in this paper. 
\begin{definition}[Tikhonov trajectory]\label{def:Tikh_traj}
Consider the problem \cref{prob:main} under \cref{assum:problem_unboundedX}. Let $\{\eta_k\}$ be a sequence of strictly positive scalars {for all $k\geq 0$}, and $x^*_{\eta_k} \in X$ denote the unique solution to the regularized variational inequality problem given by $\mathrm{VI}\left(X,F+\eta_k\nabla f\right)$. Then, the sequence $\left\{x^*_{\eta_k}\right\}$ is defined as the {\it Tikhonov trajectory} associated with the problem \cref{prob:main}. 
\end{definition}
\begin{remark}\label{rem:Tikh_traj}
	The uniqueness of the solution of $\text{VI}\left(X,F+\eta_k\nabla f\right)$ in \cref{def:Tikh_traj} is due to the strong monotonicity of the mapping $F+\eta_k\nabla f$ and closedness and convexity of the set $X$ (see Theorem 2.3.3 in \cite{FacchineiPang2003}). \cref{def:Tikh_traj} generalizes the notion of Tikhonov trajectory provided in~\cite{FacchineiPang2003} in the following way: in  \cite{FacchineiPang2003}, $x^*_{\eta_k}$ is defined as the solution to the regularized problem $\text{VI}\left(X,F+\eta_k\mathbf{I}_n\right)$. This is indeed the special case where we choose $f(x) := \frac{1}{2}\|x\|^2$ in \cref{def:Tikh_traj}. 
\end{remark}
To analyze the convergence, we utilize the properties of the Tikhonov trajectory. The following result ascertains the asymptotic convergence of {this} trajectory to the optimal solution of the problem \cref{prob:main}. It also provides an upper bound on the error between any two successive vectors of the trajectory. 
\begin{lemma}\label{Lemma 4.5}
	Consider \cref{def:Tikh_traj} and let \cref{assum:problem_unboundedX} hold. Let $\{\eta_k\}$ be {a} sequence such that $\lim_{k\to \infty}\eta_k=0$ and $\eta_k>0$ for all $k\geq 0$. Then: 
	
\noindent 	(a) The Tikhonov trajectory $\{ x^*_{\eta_{{k}}} \}$ converges to a unique limit point, that is $x^*$.
	
\noindent (b) There exists $\bar C_f>0$ such that $\left\|x^*_{\eta_{{k}}}-x^*_{\eta_{{k-1}}}\right\| \leq \frac{\bar C_f}{\mu_f} \left|1-\frac{\eta_{k-1}}{\eta_{{k}}}\right|$ for all $k\geq 1$.
\end{lemma}
\begin{proof}
See \cref{app:conv_and_rate_tikh}.
\end{proof}
The following lemmas will be employed to establish the asymptotic convergence result.
\begin{lemma}[Theorem 6, page 75 in~\cite{Knopp_1951}]\label{lem:convergence_sum}
	Let $\{u_t\}\subset \mathbb{R}^n$ denote a sequence of vectors where $\lim_{t \to \infty}u_t=\hat{u}$. Also, let $\{\alpha_k\}$ denote a sequence of strictly positive scalars such that $\sum_{k=0}^{\infty} \alpha_k = \infty$. Suppose $v_k\in \mathbb{R}^n$ is defined by {$ v_k \triangleq  \frac{\sum_{t=0}^{k}\alpha_t u_t}{\sum_{t=0}^{k}\alpha_t}$ for all $k\geq 0$}. Then{,} $ \lim\limits_{k\rightarrow\infty}v_k = \hat{u}$.
\end{lemma}
\begin{lemma}[Lemma 10, page 49 in~\cite{Polyak1987}] \label{lem:as_mean_conv_Polyak} 
	Let $\{v_k\}$ be a sequence of nonnegative random variables, where {$\EXP{v_0}<\infty$},  and let $\{\alpha_k\}$ and $\{\beta_k\}$ be deterministic scalar sequences such that $\EXP{v_{k+1}|v_0, \ldots, v_k} \leq (1-\alpha_k)v_k + \beta_k$ for all $ k \geq 0$, $0 \leq \alpha_k \leq 1$, $ \beta_k \geq 0$, $\sum_{k=0}^{\infty} \alpha_k = \infty$, $ \sum_{k=0}^{\infty} \beta_k < \infty${, and} $\lim_{k\to\infty}  \frac{\beta_k}{\alpha_k} = 0$. 	
	Then, $v_k \rightarrow 0$ almost surely and { $ \lim\limits_{k\rightarrow \infty} \EXP{v_k} = 0$}. 
\end{lemma}

\subsection{Convergence analysis}
As a key step toward performing the convergence analysis for \cref{alg:aRB-IRG} {when the set $X$ is unbounded}, next we derive a recursive inequality for the distance between the generated sequence $\{x_k\}$ by the algorithm and the Tikhonov trajectory $\{x^*_{\eta_k}\}$. To this end, we first make the following assumption: 
\begin{assumption}\label{assum:L_C}
Consider the problem \cref{prob:main} under the following assumptions:

\noindent (a) There {exist} nonnegative scalars $L_F$ {and} $B_F$ such that for all $x,y \in X$:
 $$\|F(x)-F(y)\|^2 \leq L_F^2\|x-y\|^2+B_F.$$

\noindent (b)  The gradient mapping $\nabla f$ is Lipschitz with parameter $L_f>0$.
\end{assumption}
\begin{remark}\label{rem:L_C}
By allowing $L_F$ {or} $B_F$ to {be zero}, \cref{assum:L_C} provides a unifying structure for considering both smooth and nonsmooth cases. In particular, when $L_F=0$, part (a) refers to a  bounded, but possibly non-Lipschitzian mapping $F$. Also, when $B_F=0$, part (a) refers to a Lipschitzian, but possibly unbounded mapping $F$. 
\end{remark}
The following recursive relation will play a key role in establishing the convergence.
\begin{lemma}[A recursive error bound for \cref{alg:aRB-IRG}]  \label{lem:recur_err_RBIRG}
Consider the sequence $\{x_k\}$ in \cref{alg:aRB-IRG}. Let \cref{assum:problem_unboundedX}, \cref{assum:random sample}, and \cref{assum:L_C} hold. Suppose $\{\gamma_k\}$ and $\{\eta_k\}$ are nonincreasing and strictly positive where $\lim_{k\to \infty}\eta_k=0$ and $\frac{\gamma_k}{\eta_k} \leq \frac{\mu_f\prob{{min}}}{2\prob{{max}}\left(L_F^2+\eta_0^2L_f^2\right)}$ for all $k\geq 0$. Then, for all $k\geq 1$:
		\begin{align}
	\EXP{\mathcal{D} \left(x_{k+1}, x^*_{\eta_k}\right)|\mathcal{F}_k}\leq & \frac{\prob{{max}}}{\prob{{min}}} \left( 1 -\frac{\prob{{min}} \mu_f \gamma_k \eta_k}{2}\right)\mathcal{D} \left(x_{k}, x^*_{\eta_{k-1}}\right)\notag \\ 
	&+ \frac{ \bar C_f^2\left(\mu_f \gamma_0 \eta_0+2/\prob{{min}}\right)}{ \mu_f^3\,\prob{{min}}\gamma_k \eta_k}  \left(\frac{\eta_{k-1}}{\eta_k} -1 \right)^2+ 2\gamma_k^2B_F.  \label{eqn:rec_err_ineq_in_lemma}
	\end{align}
\end{lemma}
\begin{proof}
 From \cref{def:err_D}, we have:
	\begin{align}\label{eqn:D_split_1}
	\mathcal{D} \left(x_{k+1}, x^*_{\eta_k}\right)  =   \prob{{i_k}}^{-1} \left\Vert\left.  x_{k+1}^{(i_k)}-  x_{\eta_{k}}^{*^{(i_k)}} \right\Vert\right.^2+\sum_{i=1, \thinspace	 i \neq i_k}^{d} \prob{i}^{-1} \left\Vert\left.  x_{k}^{(i)}-  x_{\eta_{k}}^{*^{(i)}} \right\Vert\right.^2.
	\end{align}
	Next, we find a bound on the term $\left\Vert\left.  x_{k+1}^{(i_k)}-  x_{\eta_{k}}^{*^{(i_k)}} \right\Vert\right.^2$. From the properties of the natural map ({cf.} Proposition 1.5.8 in \cite{FacchineiPang2003}), \cref{def:reg_map}, and that $x^*_{\eta_k} \in X$, we have $
	x^*_{\eta_k} = \mathcal{P}_X\left(x^*_{\eta_k}-\gamma_kG_k\left(x^*_{\eta_k}\right)\right)$. From \cref{assum:problem_unboundedX}(a) and that $x^*_{\eta_k} \in \text{SOL}\left(X,G_k\right) \subseteq X$, we have $x^{*^{(i_k)}}_{\eta_k} \in X_{i_k}$. Invoking the nonexpansiveness property of the projection mapping, \cref{equ:update_rule_aRBIRG},  and the preceding relation, we obtain:
		\begin{align*} 
	\left\Vert\left. x_{k+1}^{(i_k)}-  x_{\eta_{k}}^{*^{(i_k)}}\right\Vert\right.^2 \leq\left\Vert\left.  x_{k}^{(i_k)}-\gamma_{k}G_{k,i_k} \left( x_{k}\right) -x_{\eta_k}^{*^{(i_k)}}+\gamma_kG_{k,i_k}\left(x^*_{\eta_k}\right)  \right\Vert\right.^2.
	\end{align*}
Combining the preceding relation with \cref{eqn:D_split_1}, we obtain:
	\begin{align*}
	 	\mathcal{D}\left(x_{k+1}, x^*_{\eta_k}\right)& \leq 
	 	 \sum_{i=1, \thinspace	 i \neq i_k}^{d} \prob{i}^{-1} \left\Vert\left.  x_{k}^{(i)}-  x_{\eta_{k}}^{*^{(i)}} \right\Vert\right.^2 +  \prob{{i_k}}^{-1}\left\Vert\left.  x_{k}^{(i_k)}-  x_{\eta_{k}}^{*^{(i_k)}} \right\Vert\right.^2 \nonumber \\
& - 2\,\prob{{i_k}}^{-1}\gamma_k\left( x^{(i_k)}_{k} - x^{*^{(i_k)}}_{\eta_k}\right)^T\left(G_{k,i_k}\left(x_k\right)-G_{k,i_k}\left(x^*_{\eta_k}\right)\right)\nonumber \\ &+ \prob{{i_k}}^{-1}\gamma_k^2 \left\Vert\left.G_{k,i_k}\left(x_k\right)-G_{k,i_k}\left(x^*_{\eta_k}\right)\right\Vert\right.^2.
	\end{align*}
Invoking \cref{def:err_D}, from the preceding relation we obtain:	
		\begin{align*}
	 	\mathcal{D}\left(x_{k+1}, x^*_{\eta_k}\right) & \leq  \mathcal{D}\left(x_{k}, x^*_{\eta_k}\right) - 2\,\prob{{i_k}}^{-1}\gamma_k\left( x^{(i_k)}_{k} - x^{*^{(i_k)}}_{\eta_k}\right)^T\left(G_{k,i_k}\left(x_k\right)-G_{k,i_k}\left(x^*_{\eta_k}\right)\right)\nonumber \\ &+ \prob{{i_k}}^{-1}\gamma_k^2 \left\Vert\left.G_{k,i_k}\left(x_k\right)-G_{k,i_k}\left(x^*_{\eta_k}\right)\right\Vert\right.^2.
	\end{align*}
	Taking the conditional expectation from the both sides of preceding relation and noting that $\mathcal{D}\left(x_k,x^*_{\eta_{{k}}}\right)$ is $\mathcal{F}_k$--measurable, we obtain the following inequality:
			\begin{align}\label{eqn:D_split_2}
	\EXP{\mathcal{D} \left(x_{k+1}, x^*_{\eta_k}\right)|\mathcal{F}_k} & \leq  \mathcal{D}\left(x_{k}, x^*_{\eta_k}\right) + \gamma_k^2\EXP{\prob{{i_k}}^{-1} \left\Vert\left.G_{k,i_k}\left(x_k\right)-G_{k,i_k}\left(x^*_{\eta_k}\right)\right\Vert\right.^2} \notag\\ 
	&- 2\gamma_k\EXP{\prob{{i_k}}^{-1}\left( x^{(i_k)}_{k} - x^{*^{(i_k)}}_{\eta_k}\right)^T\left(G_{k,i_k}\left(x_k\right)-G_{k,i_k}\left(x^*_{\eta_k}\right)\right)}.
	\end{align}
	Next, we estimate the second and third expectations in the preceding relation:
				\begin{align}\label{eqn:D_split_3}
	&\EXP{\prob{{i_k}}^{-1}\left( x^{(i_k)}_{k} - x^{*^{(i_k)}}_{\eta_k}\right)^T\left(G_{k,i_k}\left(x_k\right)-G_{k,i_k}\left(x^*_{\eta_k}\right)\right)}\notag\\
	&=\sum_{i=1}^d\prob{{i}}\prob{{i}}^{-1}\left( x^{(i)}_{k} - x^{*^{(i)}}_{\eta_k}\right)^T\left(G_{k,i}\left(x_k\right)-G_{k,i}\left(x^*_{\eta_k}\right)\right)\notag \\
	& =\left(x_k-x^*_{\eta_k}\right)^T\left(G_k(x_k)-G_k\left(x^*_{\eta_k}\right)\right).
	\end{align}
We can also write:
\begin{align}\label{eqn:D_split_4}
	&\EXP{\prob{{i_k}}^{-1} \left\Vert\left.G_{k,i_k}\left(x_k\right)-G_{k,i_k}\left(x^*_{\eta_k}\right)\right\Vert\right.^2}\notag\\
	&=\sum_{i=1}^d\prob{{i}}\prob{{i}}^{-1}\left\Vert\left.G_{k,i}\left(x_k\right)-G_{k,i}\left(x^*_{\eta_k}\right)\right\Vert\right.^2=
	\left\Vert\left.G_{k}\left(x_k\right)-G_{k}\left(x^*_{\eta_k}\right)\right\Vert\right.^2.
	\end{align}
From \cref{assum:L_C}, taking into account that $G_k$ is $(\eta_k\mu_f)$--strongly monotone, and combining \cref{eqn:D_split_2}, \cref{eqn:D_split_3}, and \cref{eqn:D_split_4} we obtain:
	\begin{align*}
	\EXP{\mathcal{D} \left(x_{k+1}, x^*_{\eta_k}\right)|\mathcal{F}_k} & \leq  \mathcal{D}\left(x_{k}, x^*_{\eta_k}\right) - 2\mu_f\gamma_k\eta_k\left\|x_k-x^*_{\eta_k}\right\|^2\notag\\ 
	&+ 2\gamma_k^2\left(\left(L_F^2+\eta_k^2L_f^2\right)\left\|x_k-x^*_{\eta_k}\right\|^2+B_F\right).
	\end{align*}
		From \cref{rem:err_D} {and that $\{\eta_k\}$ is a nonincreasing sequence}, we obtain:
		\begin{align*}
	\EXP{\mathcal{D} \left(x_{k+1}, x^*_{\eta_k}\right)|\mathcal{F}_k} & \leq  \left(1- 2\mu_f\gamma_k\eta_k\prob{{min}}+2\gamma_k^2\prob{{max}}\left(L_F^2+{\eta_0}^2L_f^2\right)\right)\mathcal{D}\left(x_{k}, x^*_{\eta_k}\right)\notag\\
	&+ 2\gamma_k^2B_F.
	\end{align*}
	From the assumption $\gamma_k \leq \frac{\mu_f\eta_k\prob{{min}}}{2\prob{{max}}\left(L_F^2+\eta_0^2L_f^2\right)}$ and the preceding inequality, we obtain:
		\begin{align}\label{eqn:D_split_5}
	\EXP{\mathcal{D} \left(x_{k+1}, x^*_{\eta_k}\right)|\mathcal{F}_k} & \leq  \left(1- \mu_f\gamma_k\eta_k\prob{{min}}\right)\mathcal{D}\left(x_{k}, x^*_{\eta_k}\right)+ 2\gamma_k^2B_F.
	\end{align}
The preceding relation is not yet fully recursive as the term $x^*_{\eta_k}$ {on the right-hand side} must change to $x^*_{\eta_{k-1}}$. Next, we find an upper bound for $\mathcal{D}\left(x_{k}, x^*_{\eta_k}\right)$ in terms of $\mathcal{D}\left(x_{k}, x^*_{\eta_{k-1}}\right)$. Note that we have $\|u+v\|^2\leq (1+\theta)\|u\|^2+\left(1+\frac{1}{\theta}\right)\|v\|^2$ for any vectors $u,v \in \mathbb{R}^n$ and $\theta>0$. Utilizing this inequality, by setting {$u:= x_k - x^*_{\eta_{k-1}} $, $v:=x^*_{\eta_{k-1}} - x^*_{\eta_k}$}, and $\theta:= \frac{\prob{{min}}\mu_f \gamma_k \eta_k}{2}$ we obtain:
	\begin{align*}
	\left\Vert\left. x_k - x^*_{\eta_k} \right\Vert\right.^2 &\leq \left( 1 + \frac{\prob{{min}}\mu_f \gamma_k \eta_k}{2} \right) \left\Vert\left. x_k - x^*_{\eta_{k-1}} \right\Vert\right.^2\\ & + \left( 1 + \frac{2}{\prob{{min}}\mu_f \gamma_k \eta_k}  \right) \left\Vert\left. x^*_{\eta_{k-1}} - x^*_{\eta_k} \right\Vert\right.^2 .
	\end{align*}
Together with \cref{Lemma 4.5}(b) and \cref{rem:err_D}, we have:
\begin{align*}
	\prob{{min}} \,\mathcal{D} \left(x_{k}, x^*_{\eta_k}\right) &\leq \left( 1 + \frac{\prob{{min}} \mu_f \gamma_k \eta_k}{2} \right)   
	\prob{{max}} \, \mathcal{D} \left(x_{k}, x^*_{\eta_{k-1}}\right)\\ &+ \left( 1 + \frac{2}{\prob{{min}}\mu_f \gamma_k \eta_k}  \right) \frac{\bar C_f^2}{\mu_f^2}\left(1- \frac{\eta_{k-1}}{\eta_k} \right)^2.
	\end{align*}
	Dividing both sides by $\prob{{min}}$
	and substituting this in \cref{eqn:D_split_5}, we obtain:
	\begin{align*}
	 \EXP{\mathcal{D} \left(x_{k+1}, x^*_{\eta_k}\right)|\mathcal{F}_k}&\leq \frac{\prob{{max}}}{\prob{{min}}} \left(1-\gamma_k \eta_k \mu_f \prob{{min}} \right) \left( 1 +\frac{\prob{{min}} \mu_f \gamma_k \eta_k}{2}\right)\mathcal{D} \left(x_{k}, x^*_{\eta_{k-1}}\right) \\ 
	&+\frac{\bar C_f^2}{\mu_f^2\,\prob{{min}}}\left( 1 + \frac{2}{\prob{{min}} \mu_f \gamma_k \eta_k}  \right) \left(1- \frac{\eta_{k-1}}{\eta_k}  \right)^2+ 2\gamma_k^2B_F.  
	\end{align*}
\cref{eqn:rec_err_ineq_in_lemma} is obtained by noting that 
$\left(1-\gamma_k \eta_k \mu_f \prob{{min}} \right) \left( 1 +\frac{\prob{{min}} \mu_f \gamma_k \eta_k}{2}\right) \leq 1 -\frac{\prob{{min}} \mu_f \gamma_k \eta_k}{2}$.
\end{proof}

In the following result, we provide a class of update rules for the stepsize and the regularization sequences such that \cref{alg:aRB-IRG} attains both an almost sure convergence and a convergence in the mean sense. 
\begin{theorem}[Convergence of \cref{alg:aRB-IRG} when $X$ is unbounded]\label{thm:convergence_RB-IRG}\label{thm:conv_RB-IG} Consider \cref{prob:main}. Let the sequence {$\{\bar x_k\}$} be generated by \cref{alg:aRB-IRG}. Let \cref{assum:problem_unboundedX}, \cref{assum:random sample}, and \cref{assum:L_C} hold. Suppose the random block-coordinate $i_k$ in \cref{assum:random sample} is drawn from a uniform distribution {for all $k \geq 0$}. Let the stepsize $\{\gamma_k\}$ and the regularization parameter $\{ \eta_k\}$ be given by $\gamma_k := \gamma_0(k+1)^{-a}$ and $\eta_k := \eta_0(k+1)^{-b}$, respectively, where {$\gamma_0>0$, $\eta_0>0$,} $0< b <0.5< a$, and $a+b < 1$. Then, the following results hold {for all $0\leq r<1$}:

\noindent (i) The sequence $\{\bar x_k\}$ converges almost surely to the unique optimal solution of \cref{prob:main}.

\noindent (ii) We have that $\lim_{k\to \infty}\EXP{\|\bar x_k-x^*\|}= 0$. 
\end{theorem}
\begin{proof}
The proof is done in two main steps. In the first step, we show that the non-averaged sequence $\{x_k\}$ converges to $x^*$ in an almost sure sense and that $\lim_{k\to \infty}\EXP{\|  x_k-x^*\|}= 0$. In the second step, we show that these results hold for the weighted average sequence $\{\bar x_k\}$ as well.

\noindent \textit{Step 1:} The proof of this step is done by applying \cref{lem:as_mean_conv_Polyak} to the recursive inequality \cref{eqn:rec_err_ineq_in_lemma} with $\prob{i}:=\frac{1}{d}$ for all $i \in \{1,\ldots,d\}$. The details are as follows. First, we note that from the update rules of $\gamma_k$ and $\eta_k$ and that $a>b$, we have $\lim_{k \to \infty}\frac{\gamma_k}{\eta_k} =0$. Thus, there {exists} an integer $k_0\geq 1$ such that for all $k\geq k_0$, we have $\frac{\gamma_k}{\eta_k} \leq \frac{\mu_f\prob{{min}}}{2\prob{{max}}\left(L_F^2+\eta_0^2L_f^2\right)}$. This implies that the conditions of \cref{lem:recur_err_RBIRG} are satisfied and the inequality \cref{eqn:rec_err_ineq_in_lemma} holds for all $k \geq k_0$. To apply \cref{lem:as_mean_conv_Polyak}, we define the following terms for all {$k\geq 1$}:
		\begin{align*}
		&v_k \triangleq  \mathcal{D} \left(x_{k}, x^*_{\eta_{k-1}}\right), \quad \alpha_k \triangleq \frac{\mu_f\gamma_k\eta_k}{2d},\\&\beta_k \triangleq   \left( \frac{d\bar C_f^2\left(\mu_f\eta_0\gamma_0+2d\right)}{ \mu_f^3 \gamma_k \eta_k}  \right) \left( \frac{\eta_{k-1}}{\eta_k} - 1 \right)^2 + 2\gamma^2_kB_F. 
		\end{align*}
Since $\gamma_k\eta_k \to 0$, there {exists} an integer $k_1\geq k_0$ such that for any $k \geq k_1$ we have $0\leq \alpha_k \leq 1$. From the assumption that $a+b<1$, we have that $\sum_{k = k_1}^\infty\alpha_k = \infty$. Next, we show that $\sum_{k=k_1}^\infty\beta_k < \infty$. From the update rules of $\gamma_k \text{ and } \eta_k$ and invoking the Taylor series expansion, for $k \geq 2$ we can write:
		\begin{align*}
&\frac{\eta_{k-1}}{\eta_k} - 1  = \left( 1+\frac{1}{k} \right)^b -1= \left( 1+ \frac{b}{k} + \frac{b(b-1)}{2!} \frac{1}{k^2} + \frac{b(b-1)(b-2)}{3!} \frac{1}{k^3} + ... \right)-1\\
& =\frac{b}{k}\left(1-\frac{(1-b)}{2!k}+\frac{(1-b)(2-b)}{3!k^2}-\frac{(1-b)(2-b)(3-b)}{4!k^3}+...\right)\leq \frac{b}{k}\sum_{i=0}^\infty\frac{1}{k^{2i}},
\end{align*}
where the inequality is obtained using $b<1$ and neglecting the negative terms. This implies that $\frac{\eta_{k-1}}{\eta_k} - 1\leq \frac{b}{k(1-k^{-2})}$ and thus $\left(\frac{\eta_{k-1}}{\eta_k} - 1\right)^2\leq \left(\frac{4b}{3k}\right)^2 \leq\frac{2b^2}{k^2}$ for all $k\geq 2$. {Using the preceding relation, invoking} the definition of $\beta_k${,} and the update formulas of $\gamma_k$ and $\eta_k$, we have that $\beta_k =\mathcal{O}\left(k^{-(2-a-b)}\right)+\mathcal{O}\left(k^{-2a}\right)$. From the assumptions on $a$ and $b$, we obtain that $\sum_{k=k_1}^\infty\beta_k < \infty$. Also, from the assumption $a>b$, we get $\lim_{k \to \infty} \beta_k/\alpha_k = 0$. Therefore, all conditions of \cref{lem:as_mean_conv_Polyak} are satisfied. As such, we have that $\mathcal{D} \left(x_{k}, x^*_{\eta_{k-1}}\right) {\to}\ 0$ almost surely and also $\lim_{k \to \infty}\EXP{\mathcal{D} \left(x_{k}, x^*_{\eta_{k-1}}\right)}=0$. From \cref{rem:err_D} and that $i_k$ is drawn uniformly, we obtain: 
\begin{align}\label{eqn:a_s_mean_conv_last_ineq}
\left\|x_{k}- x^*\right\|^2 
&\leq 2\left\|x_{k}- x^*_{\eta_{k-1}}\right\|^2+2\left\|x^*_{\eta_{k-1}}-x^*\right\|^2\nonumber \\ &
=\frac{2}{d}\mathcal{D} \left(x_{k}, x^*_{\eta_{k-1}}\right)+2\left\|x^*_{\eta_{k-1}}-x^*\right\|^2,
\end{align}
where the first inequality is obtained from the triangle inequality. Taking the limit from both sides of the preceding relation when $k \to \infty$ and invoking \cref{Lemma 4.5}(a), we obtain $\lim_{k \to \infty}\left\|x_{k}- x^*\right\|^2 \leq \frac{2}{d}\lim_{k \to \infty}\mathcal{D} \left(x_{k}, x^*_{\eta_{k-1}}\right)$. From the almost sure convergence of $\mathcal{D} \left(x_{k}, x^*_{\eta_{k-1}}\right)$ to zero, we conclude that {$\{x_k\}$} converges to $x^*$ almost surely. To show the convergence in mean, let us take the expectation from both sides of \cref{eqn:a_s_mean_conv_last_ineq}. Noting that the Tikhonov trajectory is deterministic, we obtain that $\EXP{\left\|x_{k}- x^*\right\|^2} \leq \frac{2}{d}\EXP{\mathcal{D} \left(x_{k}, x^*_{\eta_{k-1}}\right)}+2\left\|x^*_{\eta_{k-1}}-x^*\right\|^2$. Now, taking the limit from both sides of the preceding relation when $k \to \infty$, invoking \cref{Lemma 4.5}(a), and recalling $\lim_{k \to \infty}\EXP{\mathcal{D} \left(x_{k}, x^*_{\eta_{k-1}}\right)}=0$, we conclude that $\lim_{k \to \infty}\EXP{\left\|x_{k}- x^*\right\|^2}=0$. Invoking Jensen's inequality, we can conclude that $\lim_{k \to \infty}\EXP{\left\|x_{k}- x^*\right\|}=0$. 

\noindent \textit{Step 2:} Invoking \cref{Lemma 2.10} and using the triangle inequality, we have:
\begin{align}\label{eqn:xbar_ineq_proof}
\left\|\bar x_k-x^*\right\| = \left\|\sum_{t=0}^k \lambda_{t,k}x_t-x^*\right\| =\left\|\sum_{t=0}^k \lambda_{t,k}\left(x_t-x^*\right)\right\|\leq 
\sum_{t=0}^k\lambda_{t,k}\left\| x_t-x^*\right\|,
\end{align}
{where $\lambda_{t,k}\triangleq \gamma_t^r/\sum_{j=0}^k\gamma_j^r$.} In view of \cref{lem:convergence_sum}, let us define $u_t\triangleq \left\| x_t-x^*\right\|$,  ${v_k}\triangleq \sum_{t=0}^k\lambda_{t,k}\left\| x_t-x^*\right\|$, and $\alpha_t \triangleq \gamma_t^r$. {Note that} since $ar \leq 1$, we have $\sum_{t=0}^\infty\alpha_t ={\gamma_0^r} \sum_{t=0}^\infty(t+1)^{-ar} =\infty$. {Also,} from Step 1 we have that $\hat u \triangleq \lim_{t\to \infty}u_t =0$ in an almost sure sense. Thus, from \cref{lem:convergence_sum}, we conclude that {$\{v_k\}$} converges to zero almost surely. Thus, \cref{eqn:xbar_ineq_proof} implies that $\{\bar x_k\}$ converges to $x^*$ almost surely. Next, we apply \cref{lem:convergence_sum} {again, but} in a {slightly} different fashion to show that $\lim_{k\to \infty}\EXP{\|\bar x_k-x^*\|}= 0$. From \cref{eqn:xbar_ineq_proof}{,} we have:
\begin{align}\label{eqn:xbar_ineq2_proof}
\EXP{\left\|\bar x_k-x^*\right\|}{\leq} \sum_{t=0}^k\lambda_{t,k}\EXP{\left\| x_t-x^*\right\|}.
\end{align}
In view of \cref{lem:convergence_sum}, let us define $u_t\triangleq \EXP{\left\| x_t-x^*\right\|}$,  ${v_k}\triangleq \sum_{t=0}^k\lambda_{t,k}\EXP{\left\| x_t-x^*\right\|}$, and $\alpha_t \triangleq \gamma_t^r$. First, {note that from Step 1, we have $\hat u \triangleq \lim_{t \to \infty}u_t= 0$}. In view of 
 \cref{lem:convergence_sum}, {$\lim_{k \to \infty} v_k=0$}. Thus, from \cref{eqn:xbar_ineq2_proof}, we conclude that {$\lim_{k \to \infty}\EXP{\left\|\bar x_k-x^*\right\|}=0$}. Hence, the proof is completed.
\end{proof}

\section{Experimental results}\label{sec:experiments}
In this section, we revisit the problem of finding the best Nash equilibrium formulated as in \cref{prob:best_NE}. We consider a case where the Nash game is characterized as a Cournot competition over a network. Cournot game is one of the most extensively studied economic models for competition among multiple firms, including imperfectly competitive power markets as well as rate control over communication
networks \cite{JohariThesis,KannanShanbhag2012,FacchineiPang2003}. Consider a collection of $d$ firms who compete to sell a commodity over a network with $J$ nodes. The decision of each firm $i \in\{1, \dots, d\}$ includes variables $y_{ij} $ and  $s_{ij}$, denoting the generation and sales of the firm $i$ at the node $j$, respectively. Considering  the definitions $\ y_i \triangleq \left(y_{i1}; \dots; y_{iJ}\right)$ and $s_i \triangleq \left(s_{i1}; \dots; s_{iJ}\right)$, we can compactly denote the decision variable of the $i^{th}$ firm as $x^{(i)} \triangleq \left(y_i; s_i\right) \in \mathbb{R}^{2J}$. The goal of the $i^{th}$ firm lies in minimizing the net cost function $g_i\left(x^{(i)}, x^{(-i)}\right)$ over the network defined as follows:  
\begin{align*}
g_i\left(x^{(i)}{;} x^{(-i)}\right) \triangleq \sum_{j=1}^{{J}} c_{ij} (y_{ij})- \sum_{j=1}^{{J}} s_{ij}p_j \left(\bar{s}_j\right),
\end{align*}
 where $c_{ij}:\mathbb{R}\to \mathbb{R}$ denotes the production cost function of the firm $i$ at the node $j$, $\bar{s}_j \triangleq \sum_{i=1}^{d}s_{ij}$ denotes the aggregate sales from all the firms at the node $j$, and $p_j:\mathbb{R}\to \mathbb{R}$ denotes the price function with respect to the aggregate sales $\bar s_j$ at the node $j$. We assume that the cost functions are linear and the price functions are given as $p_j \left(\bar{s}_j\right) \triangleq \alpha_j - \beta_j \left(\bar{s}_j\right)^{\sigma}$ where $\sigma \geq 1$ and  $\alpha_j$ and $ \beta_j$ are positive scalars. Throughout, we assume that the transportation costs are negligible. We let the generation be capacitated as $y_{ij} \leq \mathcal{B}_{ij}$, where $\mathcal{B}_{ij}$ is a positive scalar for $i \in\{1, \dots, d\}$ and $j \in\{1, \dots, J\}$. Lastly, for any firm $i$, the total sales must match with the total generation. Consequently, the strategy set of the firm $i$ is given as follows:
\begin{align*}
X_i \triangleq \left \{  \left(y_i; s_i\right) \mid\sum_{j=1}^{J} y_{ij} = \sum_{j=1}^{J} s_{ij}, \quad y_{ij}, s_{ij} \geq 0, \quad y_{ij} \leq \mathcal{B}_{ij}, \ \text{ for all } j = 1, \dots, J \right \}. 
\end{align*}
Following the model \cref{prob:best_NE}, we employ the Marshallian aggregate surplus function defined as $f(x)\triangleq \sum_{i=1}^d g_i\left( x^{(i)}; x^{(-i)}   \right) $. We note that the convexity of the function $f$ is implied by $\sigma \geq 1$ and the monotonicity of mapping $F$ is guaranteed when either $\sigma =1$, or when $1<\sigma\leq 3$ and $d\leq \frac{3\sigma-1}{\sigma-1}$ (cf. section $4$ in  \cite{KannanShanbhag2012}).

{\bf The set-up:} In the experiment, we consider a Cournot game among $4$ firms over $3$ nodes. We let the slopes of the linear cost functions take values between $10$ and $50$. We assume that $\alpha_j := 50$ and $\beta_j: = 0.05$ for all $j$, $\mathcal{B}_{ij} := 120$ for all $i$ and $j$, and $\sigma := 1.01$. To report the performance of \cref{alg:aRB-IRG} in terms of the suboptimality, we plot a sample average approximation of $\EXP{f\left(\bar x_N\right)}$ using the sample size of $25$. With regard to the infeasibility, we compute a sample average approximation of $\EXP{\mathrm{GAP}\left(\bar x_N\right)}$ using the same sample size. Following \cref{rem:iter_complexity}, we use $\gamma_k:=\frac{\gamma_0}{\sqrt{k+1}}$ and $\eta_k:=\frac{\eta_0}{\sqrt[4]{k+1}}$. To select the block-coordinates in \cref{alg:aRB-IRG}, we use a discrete uniform distribution. 

\begin{table}[t]
\setlength{\tabcolsep}{0pt}
\centering{
 \begin{tabular}{c || c  c  c}
  {\footnotesize$(\gamma_0,\eta_0)= \ $}& {\footnotesize  $(0.1,0.1)$} & {\footnotesize $(0.1,1)$} & {\footnotesize $(1,0.1)$} \\
 \hline\\
\rotatebox[origin=c]{90}{{\footnotesize $\ln\left(\text{sample ave. gap}\right)$}}
&
\begin{minipage}{.27\textwidth}
\includegraphics[scale=.125, angle=0]{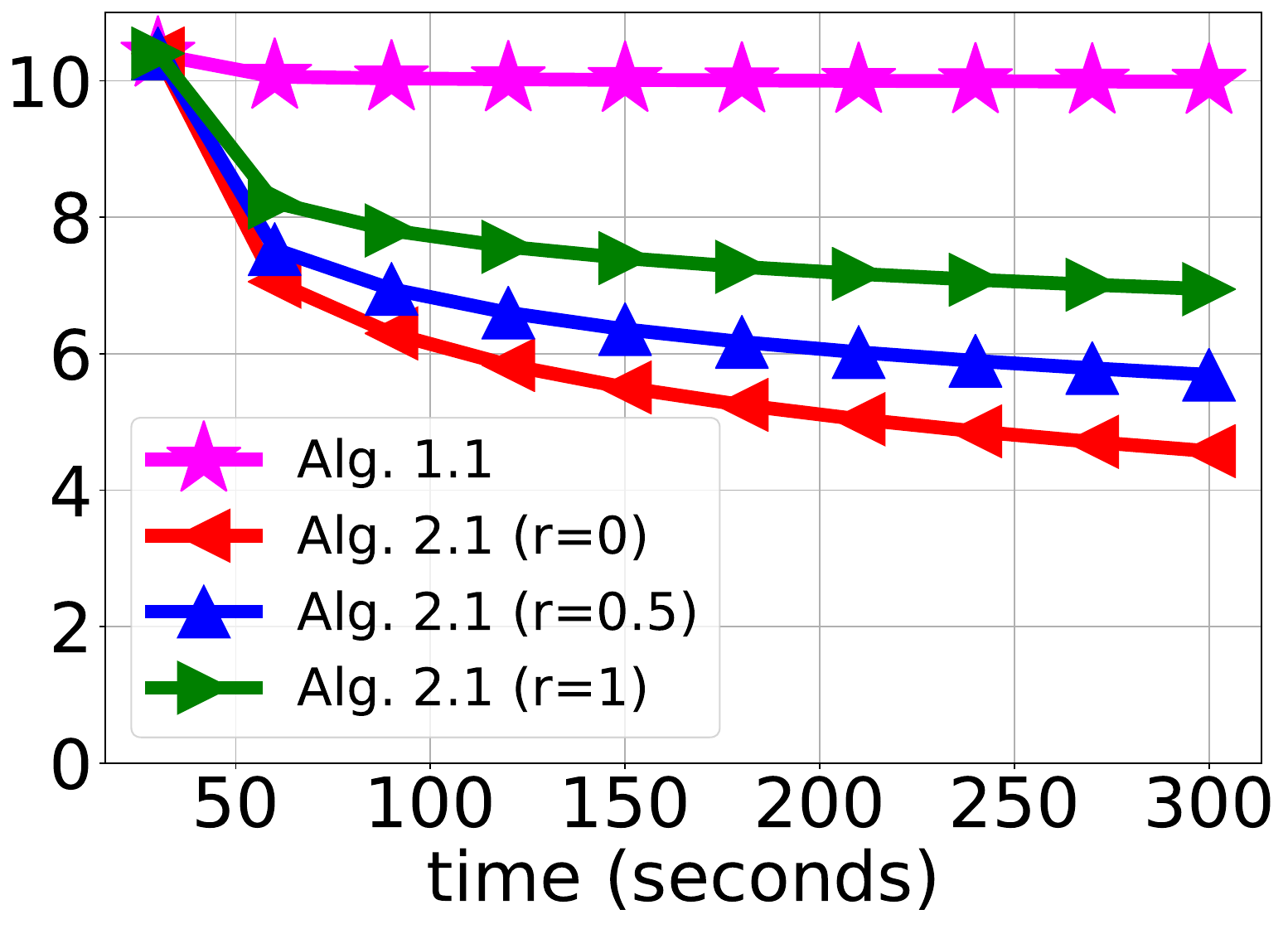}
\end{minipage}
&
\begin{minipage}{.27\textwidth}
\includegraphics[scale=.125, angle=0]{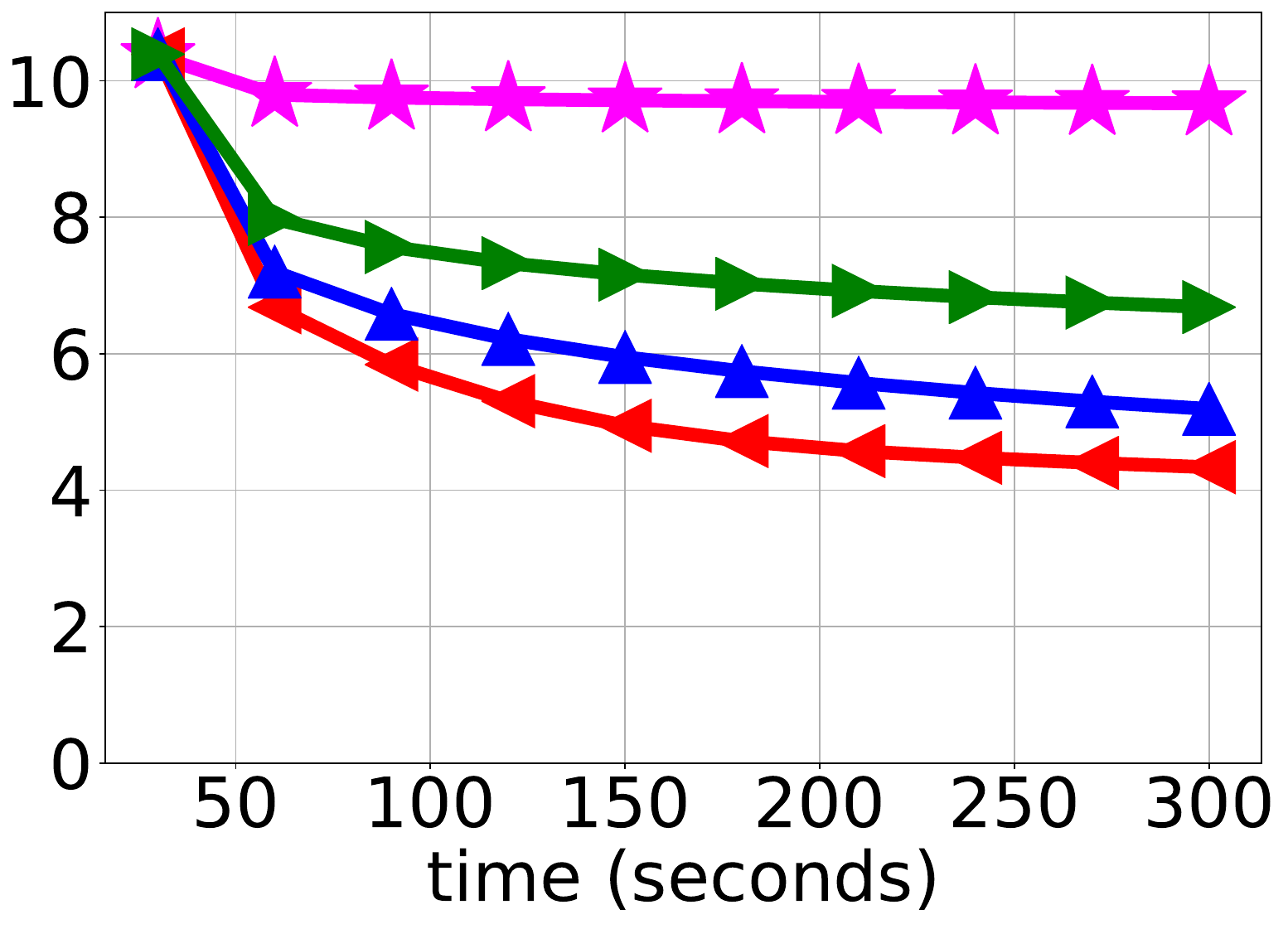}
\end{minipage}
	&
\begin{minipage}{.27\textwidth}
\includegraphics[scale=.125, angle=0]{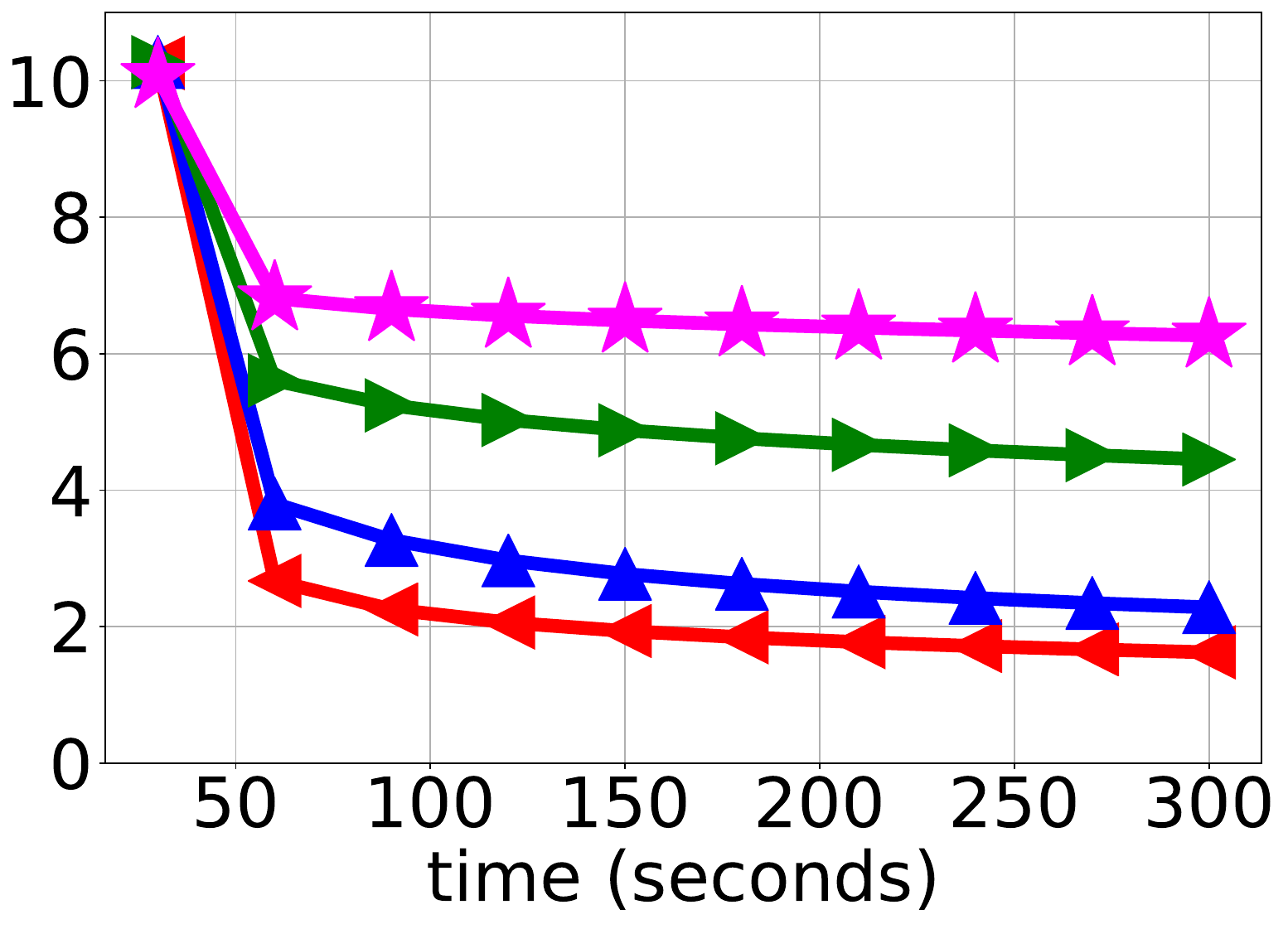}
\end{minipage}
\\
\hbox{}& & & \\
 \hline\\
\rotatebox[origin=c]{90}{{\footnotesize$\ln\left(\text{sample ave. obj.}\right)$}}
&
\begin{minipage}{.27\textwidth}
\includegraphics[scale=.125, angle=0]{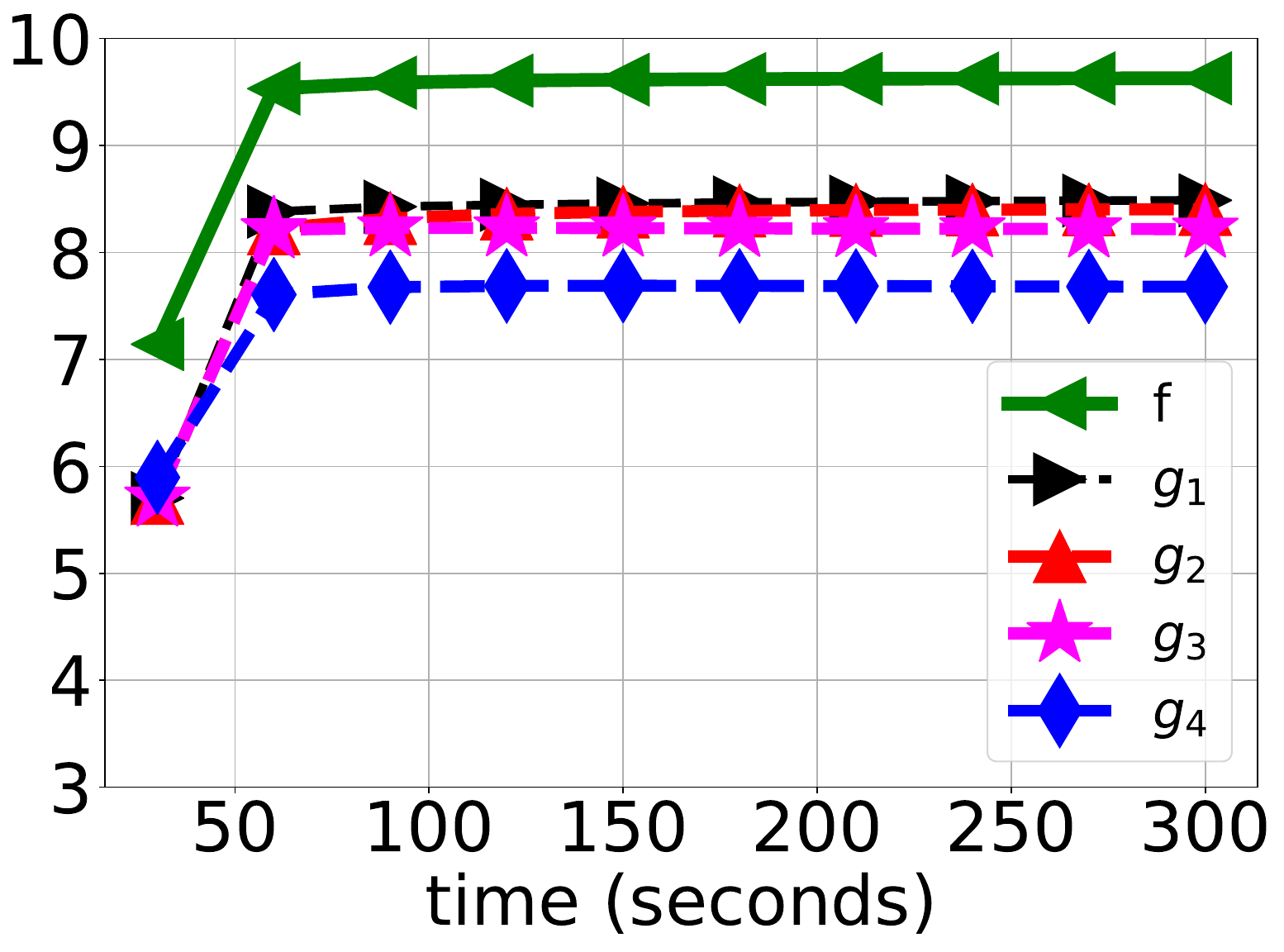}
\end{minipage}
&
\begin{minipage}{.27\textwidth}
\includegraphics[scale=.125, angle=0]{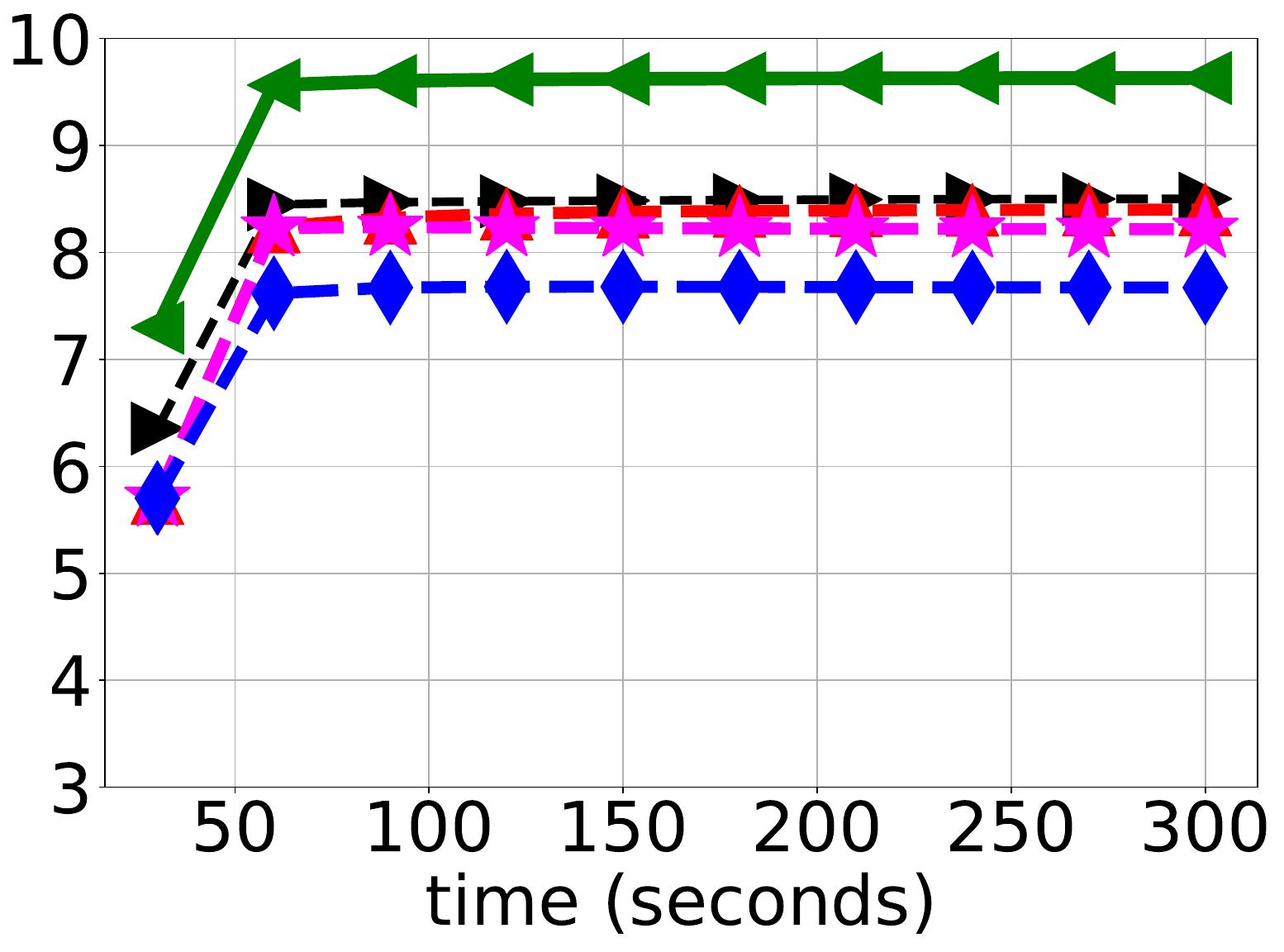}
\end{minipage}
&
\begin{minipage}{.27\textwidth}
\includegraphics[scale=.125, angle=0]{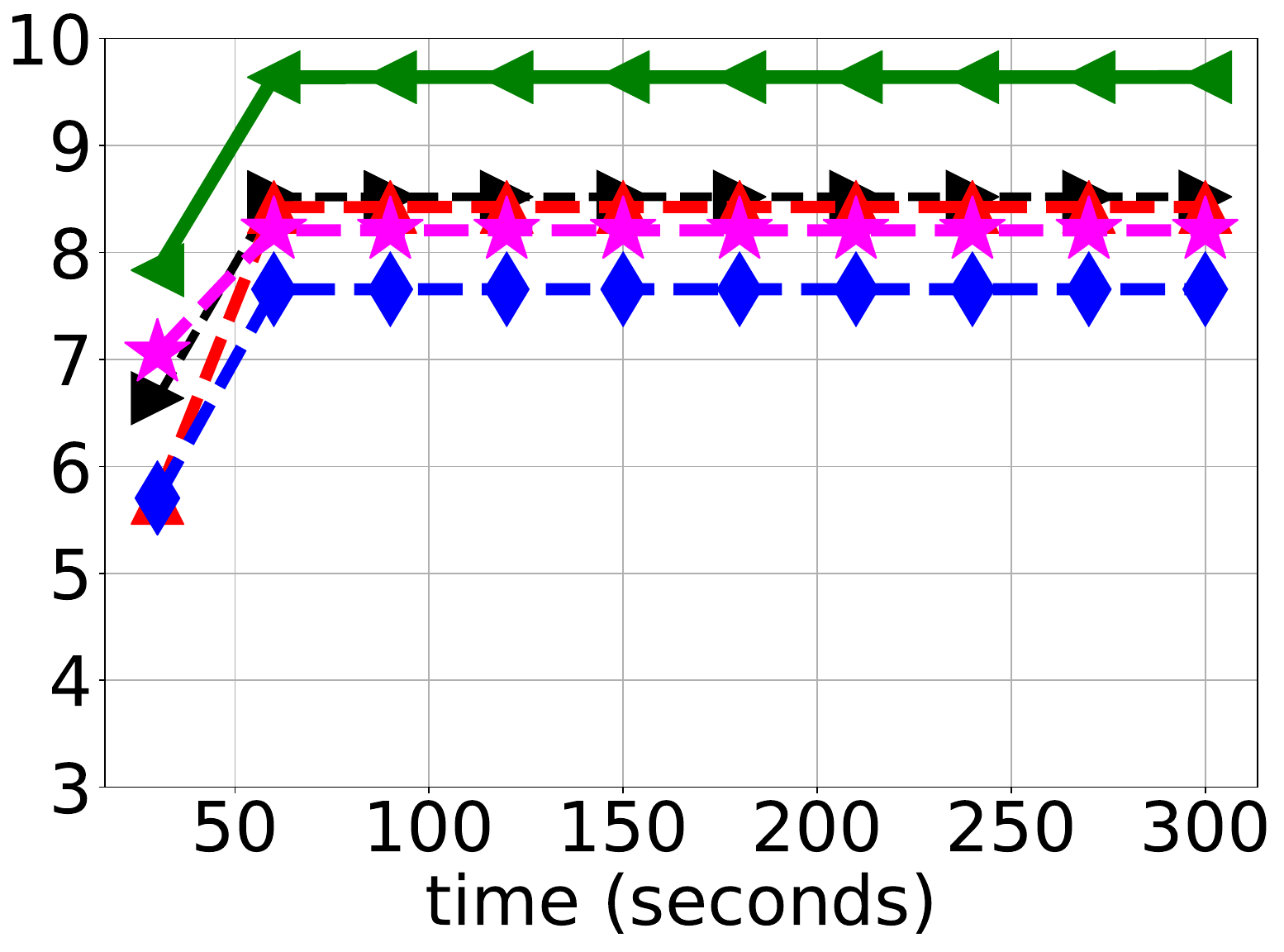}
\end{minipage}
\end{tabular}}
\captionof{figure}{\cref{alg:aRB-IRG} in terms of infeasibility and the objective function value}
\label{fig:comparison_with_SR}
\vspace{-.2in}
\end{table}

\textbf{Results and insights:}
\Cref{fig:comparison_with_SR} shows the experimental results. Here, in the top three figures, we compare the performance of \cref{alg:aRB-IRG} with that of \cref{alg:two-loop} in terms of infeasibility measured by the sample averaged gap function. Importantly, the proposed algorithm performs significantly better than the SR scheme. This claim is supported by considering the different values of the parameter $r$ and the initial conditions of the proposed scheme in terms of the initial stepsize $\gamma_0$ and the initial regularization parameter $\eta_0$. The three figures in the bottom row of \Cref{fig:comparison_with_SR} demonstrate the performance of \cref{alg:aRB-IRG} in terms of reaching a stability in the objective values. This includes the Marshallian objective function $f$ as well as the individual objective functions $g_i$. Note that all the objective values in \Cref{fig:comparison_with_SR} appear to reach to a desired level of stability after around $60$ seconds. This interesting observation could be linked to the impact of the averaging scheme \cref{eqn:ave_step_of_alg}. Generally, it is expected that the trajectories of the objective function values in \Cref{fig:comparison_with_SR} be noisy due to the randomness in the block-coordinate selection rule. However, the weighted averaging scheme employed in \cref{alg:aRB-IRG} appears to induce much robustness with respect to this uncertainty, resulting in an accelerated convergence for the proposed algorithm. 

\section{Conclusions}
\label{sec:conclusions}
Motivated by the applications arising from noncooperative multi-agent networks, we consider a class of optimization problems with Cartesian variational inequality (CVI) constraints. The computational complexity of the solution methods for addressing this class of problems appears to be unknown. We develop a single timescale algorithm equipped with non-asymptotic suboptimality and infeasibility convergence rates. Moreover, in the case where the set associated with the CVI is unbounded, we establish the global convergence of the sequence generated by the proposed algorithm. We apply the method in finding the best Nash equilibrium in a networked Cournot competition. Our experimental results show that the proposed method outperforms the classical sequential regularized schemes.

\appendix 
\section{Additional proofs}
\subsection{Proof of \cref{Lemma 1.3}}\label{app:equiv_convexprogram}
Let us define the function $\phi:\mathbb{R}^n \to \mathbb{R}$ as $\phi(x) \triangleq \frac{1}{2}\|Ax-b\|^2+\frac{1}{2}\sum_{j=1}^J \left(\max\{0,h_j(x)\}\right)^2$. We first note that $\phi$ is a differentiable function such that $\nabla \phi(x) =F(x)$ where $F$ is given by \cref{Lemma 1.3} (e.g., see page 380 in~\cite{BertsekasNLPBook2016}). Next, we also note that $\phi$ is convex. To see this, note that from the convexity of $h_j(x)$, the function $h_j^+(x)\triangleq \max\{0,h_j(x)\}$ is convex. Then, the function $\left(h_j^+(x)\right)^2$ can be viewed as a composition of $s(u)\triangleq u^2$ for $u \in \mathbb{R}$ and the convex function $h_j^+$. Since $h_j^+$ is nonnegative on its domain and $s(u)$ is nondecreasing on $[0,+\infty)$, we have that $\left(h_j^+(x)\right)^2$ is a convex function. As such, $\phi$ is a convex function as well. Consequently, from the first-order optimality conditions for convex programs, we have $\text{SOL}(X,F) = \argmin_{x \in X} \phi(x)$. To show the desired equivalence between problems \cref{prob:main} and \cref{prob:subclass_constrained_opt}, it suffices to show that $\mathcal{X}= \argmin_{x \in X} \phi(x)$ where $\mathcal{X}$ denotes the feasible set of problem \cref{prob:subclass_constrained_opt}. To show this statement, first we let $\bar x \in \mathcal{X}$. Then, from the definition of $\phi(x)$, we have $\phi(\bar x)=0$. This implies that $\bar x \in \argmin_{x \in X} \phi(x)$. Thus, we have $\mathcal{X}\subseteq  \argmin_{x \in X} \phi(x)$. Second, let $\tilde x \in \argmin_{x \in X} \phi(x)$.
The feasibility assumption of the set $\mathcal{X}$ implies that there {exists} an $x_0 \in X$ such that $Ax_0=b$ and $h_j(x_0) \leq 0$ for all $j$. This implies that $\phi(x_0) =0$. From the nonnegativity of $\phi$ and that $\tilde x \in \argmin_{x \in X} \phi(x)$, we must have $\phi(\tilde x)=0$ and $\tilde x \in X$. Therefore, we obtain $A\tilde x=b$, $h_j(\tilde x) \leq 0$ for all $j$, and $\tilde x \in X$. Thus, we have $ \argmin_{x \in X} \phi(x)\subseteq\mathcal{X} $. Hence, we conclude that $\mathcal{X} = \argmin_{x \in X} \phi(x) = \text{SOL}(X,F)$ and the proof is completed.
\subsection{Proof of \cref{Lemma 2.10}}\label{app:ave_x_second_formula}
We use induction to show $\bar{x}_{N} = \sum_{k=0}^N \lambda_{k,N} x_k$ for any $N\geq 0$. For $N=0$, the relation holds due to the initialization $\bar{x}_0 :=x_0$ in \cref{alg:aRB-IRG} and that $\lambda_{0,0}=1$. Next, let the relation hold for some $N\geq 0$. From {the hypothesis, equation \cref{eqn:ave_step_of_alg}, and} that $S_{N}=\sum_{k=0}^N\gamma_k^r$ for all $N\geq 0$, we can write: 
\begin{align*}
	\bar{x}_{N+1} = \frac{S_N\bar{x}_N + \gamma_{N+1}^r x_{N+1}}{S_{N+1}} 
	 = \frac{\sum_{k=0}^{N+1}\gamma_k^r x_k}{\sum_{k = 0}^{N+1}\gamma_k^r} = \sum_{k=0}^{N+1} \lambda_{k,N+1} x_k,
	\end{align*}
implying that the induction hypothesis holds for $N+1$. Thus, we conclude that the desired averaging formula holds for {all} $N\geq 0$. To complete the proof, note that since $\sum_{k=0}^N \lambda_{k,N}=1$, under the convexity of the set $X$, we have $\bar x_N \in X$.
\subsection{Proof of \cref{Lemma 2.12}}\label{app:deltas} (a) From \cref{def:deltas}, we can write:
\begin{align*} 
\EXP{\Delta_k\mid \mathcal{F}_k} =F(x_k)-\sum_{i=1}^d\prob{{i}}\prob{{i}}^{-1}{\mathbf{U}}_{i}F_{i}(x_k)= F(x_k) -\sum_{i=1}^d{\mathbf{U}}_{i}F_{i}(x_k) =0.
\end{align*}
The relation $\EXP{\delta_k\mid \mathcal{F}_k}=0$ can be shown in a similar fashion.

\noindent (b) We can write:
\begin{align*}
 &\EXP{\|\Delta_k\|^2\mid \mathcal{F}_k} = \sum_{i=1}^d\prob{{i}}\left\|F(x_k)-\prob{{i}}^{-1}{\mathbf{U}}_{i}F_{i}(x_k)\right\|^2\\
 &=  \sum_{i=1}^d\prob{{i}} \left(\|F(x_k)\|^2 +\prob{{i}}^{-2} \left\|{\mathbf{U}}_{i}F_{i}(x_k)\right\|^2 - 2\prob{{i}}^{-1}F(x_k)^T{\mathbf{U}}_{i}F_{i}(x_k)\right)\\
 & = \|F(x_k)\|^2 +\sum_{i=1}^d\prob{{i}}^{-1} \left\|{\mathbf{U}}_{i}F_{i}(x_k)\right\|^2  -2\sum_{i=1}^d\|F_i(x_k)\|^2 
 \leq \left(\prob{{min}}^{-1}-1\right)C_F^2.
\end{align*}
The relation $\EXP{\|\delta_k\|^2\mid \mathcal{F}_k} \leq  \left(\prob{{min}}^{-1}-1\right)C_f^2$ can be shown using a similar approach.

\subsection{Proof of \cref{Lemma 2.13}}\label{app:harmonic_bounds} Given $0\leq \alpha <1$, let us define the function $\phi:\mathbb{R}_{++} \to \mathbb{R}$ as $\phi(x) \triangleq x^{-\alpha}$ for all $x>0$. Since $\alpha >0$, the function $\phi$ is nonincreasing. We can write:
\begin{align*}
\sum_{k=0}^{N}\frac{1}{(k+1)^\alpha}
= 1 + \sum_{k=2}^{N+1}\frac{1}{k^\alpha}\leq 1+\int_{1}^{N+1}\frac{dx}{x^\alpha} =1+\frac{(N+1)^{1-\alpha}-1}{1-\alpha}\leq \frac{(N+1)^{1-\alpha}}{1-\alpha},
\end{align*}
implying the desired upper bound. To show that the lower bound holds, we can write:
\begin{align*}
\sum_{k=0}^{N}\frac{1}{(k+1)^\alpha}=\sum_{k=1}^{N+1}\frac{1}{k^\alpha}\geq \int_{1}^{N+2}\frac{dx}{x^\alpha} \geq \int_{1}^{N+1}\frac{dx}{x^\alpha}
\geq \frac{(N+1)^{1-\alpha}-0.5(N+1)^{1-\alpha}}{1-\alpha},
\end{align*}
where the last inequality is obtained using the assumption that $N \geq 2^{\frac{1}{1-\alpha}}-1$. Therefore, the desired lower bound holds as well. This completes the proof.

 \subsection{Proof of \cref{Lemma 4.5}}\label{app:conv_and_rate_tikh}
 
	\noindent (a) From the definition of $x^*$ and $x^*_{\eta_k}$ (cf. \cref{def:Tikh_traj}), we have that:
	\begin{align}
	&F(x^*)^T\left(x-x^*\right) \geq 0 \qquad \text{for all } x\in X, \label{eqn:VI_sol_ineq}	\\
	&\left(F\left(x^*_{\eta_k}\right) + \eta_{{k}}\nabla f\left(x^*_{\eta_{{k}}}\right)\right)^{T}\left(y-x^*_{\eta_{{k}}}\right)\geq 0 \qquad \text{for all } y\in X.\label{eqn:regVI_sol_ineq}
	\end{align} 
	For $x:=x^*_{\eta_k}$ and $y:=x^*$, adding the resulting two relations together, we obtain: 
		\begin{align*}
	\eta_{{k}}\nabla f\left(x^*_{\eta_{{k}}}\right)^T\left( x^*-x^*_{\eta_{{k}}} \right) \geq 	\left( F\left( x^* \right) - F\left(x^*_{\eta_{{k}}}\right) \right)^T \left( x^*-x^*_{\eta_{{k}}} \right).
	\end{align*}
From the monotonicity of the mapping $F$ and the preceding relation, we obtain that $\nabla f\left(x^*_{\eta_{{k}}}\right)^T\left( x^*-x^*_{\eta_{{k}}} \right) \geq 0$. Also, from the strong convexity of $f$, we have:
\begin{align*}
f(x^*) \geq f\left(x^*_{\eta_k}\right) +\nabla f\left(x^*_{\eta_k}\right)^T\left(x^*-x^*_{\eta_k}\right)+\frac{\mu_f}{2}\left\|x^*-x^*_{\eta_k}\right\|^2.
\end{align*}	
From the preceding relations, we {obtain:}
\begin{align}\label{eqn:x_*_x_eta_k_relation}
f(x^*) \geq f\left(x^*_{\eta_k}\right) +\frac{\mu_f}{2}\left\|x^*-x^*_{\eta_k}\right\|^2 \qquad \hbox{for all }k\geq 0.
\end{align}	
{Thus, $f(x^*) \geq f\left(x^*_{\eta_k}\right)$ for all $k\geq 0$.} Recall that from \cref{rem:unique_x_star}, under \cref{assum:problem_unboundedX}, $x^*\in X$ and $x^*_{\eta_k}\in X$ both exist and are unique. {Therefore,  $f\left(x^*_{\eta_k}\right)$ is bounded above for all $k\geq 0$. From this statement and invoking the coercive property of $f$ (implied by the strong convexity of $f$), we can conclude that} $\{x^*_{\eta_k}\}$ is a bounded sequence. Therefore, it must have at least one limit point. Let $\{x^*_{\eta_k}\}_{k \in \mathcal{K}}$ be an arbitrary subsequence such that $\lim_{k \to \infty, \ k \in \mathcal{K}}x^*_{\eta_k} = \hat x$. We show that $\hat x \in \text{SOL}(X,F)$. Taking the limit from both sides of \cref{eqn:regVI_sol_ineq} with respect to the aforementioned  subsequence and using the continuity {of} $F$ and $\nabla f$, we obtain that for all $y \in X$, 
$
\left(F\left(\hat x\right) + \lim_{k \to \infty, \ k \in \mathcal{K}}\eta_{{k}}\nabla f\left(\hat x\right)\right)^{T}\left(y-\hat x\right)\geq 0
$. Note that the mapping $\nabla f\left(\hat x\right)$ is bounded. This is because $\hat x \in X$ (due to the closedness of $X$) and that $\nabla f$ is continuous on the set $X$. Therefore, from the preceding inequality {and $\lim_{k\to \infty}\eta_k=0$}, we obtain $F\left(\hat x\right)^T\left(y-\hat x\right)\geq 0$ for all $y \in X$, implying that $\hat x \in \text{SOL}(X,F)$ and so{,} $\hat x$ is a feasible solution {to} \cref{prob:main}. Next, we show that $\hat x$ is the optimal solution to \cref{prob:main}. From \cref{eqn:x_*_x_eta_k_relation}, continuity of $f$, and neglecting the term $\frac{\mu_f}{2}\left\|x^*-x^*_{\eta_k}\right\|^2 $, we obtain $f\left(x^* \right)  \geq  f\left(\lim_{k\rightarrow \infty, \ k \in \mathcal{K}} x^*_{\eta_k} \right) = f(\hat x)$. Hence, from the uniqueness of $x^*$, all the limit points of $\{x^*_{\eta_k}\}$ fall in the singleton $\{x^*\}$ and the proof is completed.

			\noindent (b) If $x^*_{\eta_k} = x^*_{\eta_{k-1}}$, the desired relation holds. Suppose for $k\geq 1$, we have $x^*_{\eta_k} \neq x^*_{\eta_{k-1}}$. From $x^*_{\eta_{k-1}} \in \text{SOL}\left(X,F+\eta_{k-1}\nabla f\right)$ and $x^*_{\eta_k} \in \text{SOL}\left(X,F+\eta_k\nabla f\right)$, we have that:
	\begin{align*}
	&\left(F\left(x^*_{\eta_{k-1}}\right) + \eta_{{k-1}}\nabla f\left(x^*_{\eta_{{k-1}}}\right)\right)^{T}\left(x-x^*_{\eta_{{k-1}}}\right)\geq 0 \qquad \text{for all } x\in X, \\
	&\left(F\left(x^*_{\eta_k}\right) + \eta_{{k}}\nabla f\left(x^*_{\eta_{{k}}}\right)\right)^{T}\left(y-x^*_{\eta_{{k}}}\right)\geq 0 \qquad \text{for all } y\in X.
	\end{align*} 
	Adding the resulting two relations together, for $x:=x^*_{\eta_k}$ and $y:=x^*_{\eta_{k-1}}$ we have: 
		\begin{align*}
	\left(-F\left(x^*_{\eta_k}\right) - \eta_{{k}}\nabla f\left(x^*_{\eta_{{k}}}\right)+F(x^*_{\eta_{k-1}}) + \eta_{{k-1}}\nabla f\left(x^*_{\eta_{{k}-1}}\right)\right)^{T}\left(x^*_{\eta_{{k}}}-x^*_{\eta_{{k-1}}}\right)\geq 0. 
	\end{align*}
The monotonicity of $F$ implies that $\left(F\left(x^*_{\eta_k}\right)-F\left(x^*_{\eta_{k-1}}\right)\right)^T\left(x^*_{\eta_{{k}}} - x^*_{\eta_{{k}-1}}\right)\geq 0$. Adding this relation to the preceding inequality, we have:	
	\begin{align*}
	\left(\eta_{{k}}\nabla f\left(x^*_{\eta_{{k}}}\right)-\eta_{{k-1}}\nabla f\left(x^*_{\eta_{{k}-1}}\right)\right)^T\left( x^*_{\eta_{{k}-1}}-x^*_{\eta_{{k}}} \right)\geq 0.
	\end{align*}
	Adding and subtracting the term $\eta_{{k}}\nabla f\left( x^*_{\eta_{{k}-1}} \right)^T\left(x^*_{\eta_{{k-1}}}-x^*_{\eta_{{k}}} \right)$, we obtain:
	\begin{align}\label{eqn:lem_tikh_b_eqn1}
	&\left(\eta_{{k}}-\eta_{{k}-1}\right) \nabla f \left( x^*_{\eta_{{k}-1}} \right)^{T}\left(x^*_{\eta_{{k-1}}}-x^*_{\eta_{{k}}} \right) \geq \nonumber \\ &\eta_{{k}}\left(\nabla f\left(x^*_{\eta_{{k}-1}}\right) - \nabla f\left(x^*_{\eta_{{k}}}\right) \right)^T\left( x^*_{\eta_{{k}-1}}-x^*_{\eta_{{k}}} \right).
	\end{align}
From the strong convexity of function $f$, we have: 
	\begin{align}\label{eqn:lem_tikh_b_eqn2}
	\left(\nabla f \left(x^*_{\eta_{k-1}}\right)-\nabla f \left( x^*_{\eta_{{k}}} \right)\right)^T\left(x^*_{\eta_{k-1}}-x^*_{\eta_{k}}\right) \geq \mu_f \left\Vert x^*_{\eta_k} - x^*_{\eta_{{k}-1}} \right\Vert^2.
	\end{align}
	From \cref{eqn:lem_tikh_b_eqn1} and \cref{eqn:lem_tikh_b_eqn2}, we can write: 
	\begin{align*}
	\left(\eta_{{k}}-\eta_{{k}-1}\right) \nabla f \left( x^*_{\eta_{{k}-1}} \right)^{T}\left(x^*_{\eta_{{k-1}}}-x^*_{\eta_{{k}}} \right) \geq \eta_{{k}}\mu_f \left\Vert x^*_{\eta_k} - x^*_{\eta_{{k}-1}} \right\Vert^2.
	\end{align*}
Using the Cauchy-Schwarz inequality, we obtain:
	\begin{align*}
	\left|\eta_{{k}}-\eta_{{k}-1}\right|\left\Vert \nabla f \left( x^*_{\eta_{{k}-1}} \right)\right\Vert \left\Vert x^*_{\eta_{{k-1}}}-x^*_{\eta_{{k}}} \right\Vert \geq \eta_{{k}}\mu_f \left\Vert x^*_{\eta_k} - x^*_{\eta_{{k}-1}} \right\Vert^2,
	\end{align*}	
	Since $x^*_{\eta_k} \neq x^*_{\eta_{k-1}}$, dividing the both sides by $\eta_k\left\Vert x^*_{\eta_k} - x^*_{\eta_{{k}-1}} \right\Vert$, we obtain:
	\begin{align}\label{eqn:lem_tikh_b_eqn_last}
	\left|1-\frac{\eta_{k-1}}{\eta_k}\right|\left\Vert \nabla f \left( x^*_{\eta_{{k}-1}} \right)\right\Vert \geq\mu_f \left\Vert x^*_{\eta_k} - x^*_{\eta_{{k}-1}} \right\Vert.
	\end{align}
	From part (a), the trajectory $\{x^*_{\eta_k}\}$ is bounded. Also, for any $k\geq 0$, $x^*_{\eta_k} \in X$ by the definition. Since $X$ is closed, there exists a compact set $S {\subset} X$ such that $\{x^*_{\eta_k}\}\subset S$. This statement and the continuity of $\nabla f$ imply that there {exists} $\bar C_f>0$ such that $\left\Vert \nabla f \left( x^*_{\eta_{{k}-1}} \right)\right\Vert \leq \bar C_f$ for all $k\geq 1$. Thus, from \cref{eqn:lem_tikh_b_eqn_last}, we obtain the desired inequality. 

 {\subsection{Proof of \cref{cor:a-IRG_rate_results}}\label{app:a-IRG_rate_results}
Let us rewrite \cref{prob:main} as the equivalent problem: 
\begin{equation}\label{prob:main_equiv}
\begin{aligned}
& {\text{minimize}}
& & f(x) \\
& \text{subject to}
& & x \in \text{SOL}(Y,F),
\end{aligned}
\end{equation}
where $Y \triangleq \prod_{i=1}^{d'}Y_i$ and $d'\triangleq 1$ and $Y_1 \triangleq X$. Note that this setting immediately implies that $Y=X$. Now, let us consider \cref{alg:aRB-IRG} for solving \cref{prob:main_equiv} where we assume that $x_0 \in X$ is an arbitrary fixed vector. Since $d'=1$, \cref{assum:random sample} holds with $\mathrm{Prob}\left(i_k=1\right) = 1$ for all $k\geq 0$. This setting implies that \cref{alg:aRB-IRG} reduces to a deterministic scheme where the step 5 in \cref{alg:aRB-IRG} is equivalent to the following update rule:
		\begin{align}\label{equ:update_rule_aIRG_compact} 
		{x_{k+1}} :=
		\proj[X]{ x_{k} -\gamma_{k}\left(F \left( x_{k}\right) + \eta_{k} \tilde \nabla  f\left(x_{k}\right)\right)},
		\end{align}
where we used $Y=Y_1=X$. Next, we note that from the properties of the Euclidean projection mapping, for any $z \in X$ where $X \triangleq \prod_{i=1}^dX_i$, we have that $\proj[X]{z} = \prod_{i=1}^d\proj[X_i]{z^{(i)}}$. In view of this property, the equation \cref{equ:update_rule_aIRG_compact} compactly represents the $d$ updates given by \cref{equ:update_rule_aIRG}. Therefore, \cref{alg:a-IRG} is equivalent to \cref{alg:aRB-IRG} and thus, all the results in \cref{thm:a-IRG_rate_results} will hold with $\prob{{min}}=1$. Note that in both \cref{eqn:f_rate} and \cref{eqn:gap_rate}, the expectation is eliminated. This completes the proof.  }

\bibliographystyle{siamplain}
\bibliography{ref_siopt_Rev1_v01,references_Rev1_HDK}

\begin{thebibliography}{10}

\bibitem{alpcan02game}
{\sc T.~Alpcan and T.~Ba\c{s}ar}, {\em A game-theoretic framework for
  congestion control in general topology networks}, in Proceedings of the 41st
  IEEE Conference on Decision and Control, December 2002, pp.~1218--1224.

\bibitem{alpcan03distributed}
{\sc T.~Alpcan and T.~Ba\c{s}ar}, {\em Distributed algorithms for {N}ash
  equilibria of flow control games}, in Advances in Dynamic Games, vol.~7 of
  Annals of the International Society of Dynamic Games, Birkh{\"a}user Boston,
  2003, pp.~473--498.

\bibitem{FarzadOMS19}
{\sc M.~Amini and F.~Yousefian}, {\em An iterative regularized mirror descent
  method for ill-posed nondifferentiable stochastic optimization},  (2019),
  \url{https://arxiv.org/abs/1901.09506}.

\bibitem{PoS08}
{\sc E.~Anshelevich, A.~Dasgupta, J.~Kleinberg, E.~Tardos, T.~Wexler, and
  T.~Roughgarden}, {\em The price of stability for network design with fair
  cost allocation}, SIAM Journal on Computing, 38 (2008), pp.~1602–--1623.

\bibitem{BeckSabach2014}
{\sc A.~Beck and S.~Sabach}, {\em A first order method for finding minimal
  norm-like solutions of convex optimization problems}, Mathematical
  Programming, 147 (2014), pp.~25--46.

\bibitem{BertsekasNLPBook2016}
{\sc D.~P. Bertsekas}, {\em Nonlinear Programming}, Athena Scientific,
  Bellmont, MA, 3th~ed., 2016.

\bibitem{CensorGibaliReich2011}
{\sc Y.~Censor, A.~Gibali, and S.~Reich}, {\em The subgradient extragradient
  method for solving variational inequalities in {H}ilbert space}, Journal of
  Optimization Theory and Applications, 148 (2011), pp.~318--335.

\bibitem{CensorGibaliReich2012}
{\sc Y.~Censor, A.~Gibali, and S.~Reich}, {\em Extensions of {K}orpelevich's
  extragradient method for the variational inequality problem in {E}uclidean
  space}, Optimization, 61 (2012), pp.~1119--1132.

\bibitem{ChenLanOuyang2017}
{\sc Y.~Chen, G.~Lan, and Y.~Ouyang}, {\em Accelerated schemes for a class of
  variational inequalities}, {M}athematical Programming, 165 (2017),
  pp.~113--149.

\bibitem{Correa04}
{\sc J.~R. Correa, A.~S. Schulz, and N.~E. Stier-Moses}, {\em Selfish routing
  in capacitated networks}, Mathematics of Operations Research, 29 (2004),
  pp.~961--976.

\bibitem{Lan-VI-13}
{\sc C.~D. Dang and G.~Lan}, {\em On the convergence properties of
  non-{E}uclidean extragradient methods for variational inequalities with
  generalized monotone operators}, Computational Optimization and Applications,
  60 (2015), pp.~277--310.

\bibitem{Dang15}
{\sc C.~D. Dang and G.~Lan}, {\em Stochastic block mirror descent methods for
  nonsmooth and stochastic optimization}, SIAM Journal on Optimization, 25
  (2015), pp.~856--881.

\bibitem{FacchineiPang2003}
{\sc F.~Facchinei and J.-S. Pang}, {\em Finite-dimensional Variational
  Inequalities and Complementarity Problems. {V}ols. {I,II}}, Springer Series
  in Operations Research, Springer-Verlag, New York, 2003.

\bibitem{GarrigosRosascoVilla2018}
{\sc G.~Garrigos, L.~Rosasco, and S.~Villa}, {\em Iterative regularization via
  dual diagonal descent}, Journal of Mathematical Imaging and Vision, 60
  (2018), pp.~189--215.

\bibitem{IusemJofreOliveiraThompson2016}
{\sc A.~N. Iusem, A.~Jofr\'e, R.~I. Oliveira, and P.~Thompson}, {\em
  Extragradient method with variance reduction for stochastic variational
  inequalities}, SIAM Journal on Optimization, 27 (2016), pp.~686--724.

\bibitem{IusemJofreThompson2019}
{\sc A.~N. Iusem, A.~Jofr\'e, and P.~Thompson}, {\em Incremental constraint
  projection methods for monotone stochastic variational inequalities},
  {M}athematics of {O}perations {R}esearch, 44 (2019), pp.~236--263.

\bibitem{IusemNasri2011}
{\sc A.~N. Iusem and M.~Nasri}, {\em Korpelevich's method for variational
  inequality problems in {B}anach spaces}, Journal of Global Optimization, 50
  (2011), pp.~59--76.

\bibitem{JiangXu2008}
{\sc H.~Jiang and H.~Xu}, {\em Stochastic approximation approaches to the
  stochastic variational inequality problem}, IEEE Transactions on Automatic
  Control, 53 (2008), pp.~1462--1475.

\bibitem{JohariThesis}
{\sc R.~Johari}, {\em Efficiency Loss in Market Mechanisms for Resource
  Allocation}, PhD thesis, MIT, 2004.

\bibitem{Nem11}
{\sc A.~Juditsky, A.~Nemirovski, and C.~Tauvel}, {\em Solving variational
  inequalities with stochastic mirror-prox algorithm}, Stochastic Systems, 1
  (2011), pp.~17--58.

\bibitem{KannanShanbhag2012}
{\sc A.~Kannan and U.~V. Shanbhag}, {\em Distributed computation of equilibria
  in monotone {N}ash games via iterative regularization techniques}, SIAM
  Journal on Optimization, 22 (2012), pp.~1177--1205.

\bibitem{KShKim11}
{\sc A.~Kannan, U.~V. Shanbhag, and H.~M. Kim}, {\em Strategic behavior in
  power markets under uncertainty}, Energy Systems, 2 (2011), pp.~115--141.

\bibitem{KShKim12}
{\sc A.~Kannan, U.~V. Shanbhag, and H.~M. Kim}, {\em Addressing supply-side
  risk in uncertain power markets: stochastic {N}ash models, scalable
  algorithms and error analysis}, Optimization Methods and Software, 28 (2013),
  pp.~1095--1138.

\bibitem{FarzadHarshalACC19}
{\sc H.~Kaushik and F.~Yousefian}, {\em A randomized block coordinate iterative
  regularized subgradient method for high-dimensional ill-posed convex
  optimization}, in Proceedings of the American Control Conference, IEEE, July
  2019, pp.~3420--3425, \url{https://doi.org/10.23919/ACC.2019.8815256},
  \url{https://arxiv.org/abs/1809.10035}.

\bibitem{Knopp_1951}
{\sc K.~Knopp}, {\em Theory and applications of infinite series}, Blackie \&
  Son Ltd., Bishopbriggs, Glasgow G64 2NZ, Scotland, 1951.

\bibitem{korp76}
{\sc G.~M. Korpelevich}, {\em An extragradient method for finding saddle points
  and for other problems}, Eknomika i Matematicheskie Metody, 12 (1976),
  pp.~747–--756.

\bibitem{KoshalNedichShanbhag2013}
{\sc J.~Koshal, A.~Nedi\'c, and U.~V. Shanbhag}, {\em Regularized iterative
  stochastic approximation methods for stochastic variational inequality
  problems}, IEEE Transactions on Automatic Control, 58 (2013), pp.~594--609.

\bibitem{LeiShanbhag2020}
{\sc J.~Lei, U.~V. Shanbhag, J.-S. Pang, and S.~Sen}, {\em On synchronous,
  asynchronous, and randomized best-response schemes for stochastic {N}ash
  games}, Mathematics of Operations Research, 45 (2020), pp.~157--190.

\bibitem{LemkeHowson1964}
{\sc C.~E. Lemke and J.~T. {Howson Jr.}}, {\em Equilibrium points of bimatrix
  games}, Journal of the Society for Industrial and Applied Mathematics, 12
  (1964), pp.~413--423.

\bibitem{MarcotteZhu1998}
{\sc P.~Marcotte and D.~Zhu}, {\em Weak sharp solutions of variational
  inequalities}, SIAM Journal on Optimization, 9 (1998), pp.~179--189.

\bibitem{Nedich2011}
{\sc A.~Nedi\'c}, {\em Random algorithms for convex minimization problems},
  {M}athematical {P}rogramming, 129 (2011), pp.~225--253.

\bibitem{Nemirovski2004}
{\sc A.~Nemirovski}, {\em Prox-method with rate of convergence
  $\mathcal{O}(1/t)$ for variational inequalities with {L}ipschitz continuous
  monotone operators and smooth convex-concave saddle point problems}, {SIAM}
  Journal on Optimization, 15 (2004), pp.~229--251.

\bibitem{Nesterov2012}
{\sc {\relax YU}.~Nesterov}, {\em Efficiency of coordinate descent methods on
  huge-scale optimization problems}, SIAM Journal on Optimization, 22 (2012),
  pp.~341--362.

\bibitem{NisanBook2007}
{\sc N.~Nisan, T.~Roughgarden, E.~Tardos, and V.~V. Vazirani}, {\em Algorithmic
  Game Theory}, Cambridge University Press, New York, NY, USA, 2007.

\bibitem{OsborneRubinstein1994}
{\sc M.~J. Osborne and A.~Rubinstein}, {\em A Course in Game Theory}, MIT
  Press, Cambridge, Massachusetts, 1994.

\bibitem{Polyak1987}
{\sc B.~T. Polyak}, {\em Introduction to Optimization}, Optimization Software,
  Inc., New York, 1987.

\bibitem{RicktarikTakac2014}
{\sc P.~Rickt\'arik and M.~Tak\'{a}\v{c}}, {\em Iteration complexity of
  randomized block-coordinate descent methods for minimizing a composite
  function}, Mathematical Programming, 144 (2014), pp.~1--38.

\bibitem{Rockafellar98}
{\sc R.~T. Rockafellar and R.~J.-B. Wets}, {\em Variational Analysis},
  Springer-Verlag Berlin Heidelberg, 1998.

\bibitem{Rough04}
{\sc T.~Roughgarden}, {\em Stackelberg scheduling strategies}, SIAM Journal on
  Computing, 33 (2004), pp.~332--350.

\bibitem{SabachShtern2017}
{\sc S.~Sabach and S.~Shtern}, {\em A first order method for solving convex
  bilevel optimization problems}, SIAM Journal on Optimization, 27 (2017),
  pp.~640--660.

\bibitem{Scarf1967}
{\sc H.~Scarf}, {\em The approximation of fixed points of a continuous
  mapping}, SIAM Journal on Applied Mathematics, 15 (1967), pp.~1328--1343.

\bibitem{ScutariPalomarFacchineiPang2010}
{\sc G.~Scutari, D.~P. Palomar, F.~Facchinei, and J.-S. Pang}, {\em Convex
  optimization, game theory, and variational inequality theory}, IEEE Signal
  Processing Magazine, 27 (2010), pp.~35--49.

\bibitem{scutari10monotone}
{\sc G.~Scutari, D.~P. Palomar, F.~Facchinei, and J.-S. Pang}, {\em Monotone
  games for cognitive radio systems},  (2012), pp.~83--112.

\bibitem{ShwartzZhang2013}
{\sc S.~Shalev-Shwartz and T.~Zhang}, {\em Stochastic dual coordinate ascent
  methods for regularized loss minimization}, Journal of Machine Learning
  Research, 14 (2013), pp.~567--599.

\bibitem{SIGlynn11}
{\sc U.~V. Shanbhag, G.~Infanger, and P.~W. Glynn}, {\em A complementarity
  framework for forward contracting under uncertainty}, Operations Research, 59
  (2011), pp.~810--834.

\bibitem{Solodov2007}
{\sc M.~V. Solodov}, {\em An explicit descent method for bilevel convex
  optimization}, Journal of Convex Analysis, 14 (2007), pp.~227--237.

\bibitem{aldo1}
{\sc J.~Wang, G.~Scutari, and D.~P. Palomar}, {\em Robust {MIMO} cognitive
  radio via game theory}, IEEE Transactions on Signal Processing, 59 (2011),
  pp.~1183--1201.

\bibitem{wang2015}
{\sc M.~Wang and D.~P. Bertsekas}, {\em Incremental constraint projection
  methods for variational inequalities}, Mathematical Programming, 150 (2015),
  pp.~321--363.

\bibitem{XuViscosity2004}
{\sc H.-K. Xu}, {\em Viscosity approximation methods for nonexpansive
  mappings}, Journal of Mathematical Analysis and Applications, 298 (2004),
  pp.~279--291.

\bibitem{yin09nash2}
{\sc H.~Yin, U.~V. Shanbhag, and P.~G. Mehta}, {\em Nash equilibrium problems
  with scaled congestion costs and shared constraints}, IEEE Transactions on
  Automatic Control, 56 (2011), pp.~1702--1708.

\bibitem{FarzadPushPull2020}
{\sc F.~Yousefian}, {\em Bilevel distributed optimization in directed
  networks}, in Proceedings of the American Control Conference (accepted),
  2021, \url{https://arxiv.org/abs/2006.07564v2}.

\bibitem{YousefianNedichShanbhag2014}
{\sc F.~Yousefian, A.~Nedi\'c, and U.~V. Shanbhag}, {\em Optimal robust
  smoothing extragradient algorithms for stochastic variational inequality
  problems}, in Proceedings of the $53^{\text{rd}}$ IEEE Conference on Decision
  and Control, IEEE, Dec. 2014, pp.~5831--5836.

\bibitem{FarzadMathProg17}
{\sc F.~Yousefian, A.~Nedi\'c, and U.~V. Shanbhag}, {\em On smoothing,
  regularization, and averaging in stochastic approximation methods for
  stochastic variational inequality problems}, Mathematical Programming, 165
  (2017), pp.~391--431, \url{https://doi.org/10.1007/s10107-017-1175-y}.

\bibitem{FarzadSetValued18}
{\sc F.~Yousefian, A.~Nedi\'c, and U.~V. Shanbhag}, {\em On stochastic
  mirror-prox algorithms for stochastic {C}artesian variational inequalities:
  Randomized block coordinate and optimal averaging schemes}, Set-Valued and
  Variational Analysis, 26 (2018), pp.~789--819,
  \url{https://doi.org/10.1007/s11228-018-0472-9}.

\bibitem{FarzadSIOPT20}
{\sc F.~Yousefian, A.~Nedi\'c, and U.~V. Shanbhag}, {\em On stochastic and
  deterministic quasi-{N}ewton methods for nonstrongly convex optimization:
  Asymptotic convergence and rate analysis}, SIAM Journal on Optimization, 30
  (2020), pp.~1144--1172, \url{https://doi.org/10.1137/17M1152474}.

\end{thebibliography}
\end{document}